\documentclass[10pt,reqno]{amsart}
\usepackage{amsmath, amssymb, amsfonts, amsthm, xcolor}
\usepackage{mathtools}
\usepackage{enumerate}
\usepackage{hyperref}
\newtheorem{thm}{Theorem}[section]
\newtheorem{lem}[thm]{Lemma}
\newtheorem{prop}[thm]{Proposition}
\theoremstyle{definition}
\newtheorem{dfn}[thm]{Definition}
\newtheorem{exple}[thm]{Example}

\newtheorem{conj}[thm]{Conjecture}

\newtheorem{ques}[thm]{Question}

\newtheorem{remark}[thm]{Remark}
\theoremstyle{plain}
\newtheorem{cor}[thm]{Corollary}

\numberwithin{equation}{section}
\usepackage[a4paper, top = 1.2in, bottom = 1.2in, left = 1.2in, right = 1.2in]{geometry}
\numberwithin{equation}{section}

\newcommand{\N}{\mathbb{N}}
\newcommand{\Q}{\mathbb{Q}}

\newcommand{\Z}{\mathbb{Z}}
\newcommand{\F}{\mathbb{F}}

\newcommand{\mfm}{\mathfrak{m}}
\newcommand{\mfn}{\mathfrak{n}}
\newcommand{\mfp}{\mathfrak{p}}

\newcommand{\mfr}{\mathfrak{r}}

\newcommand{\m}{\mathfrak{m}}
\newcommand{\n}{\mathfrak{n}}
\newcommand{\p}{\mathfrak{p}}
\newcommand{\q}{\mathfrak{q}}

\newcommand{\GL}{\mathrm{GL}}
\newcommand{\Ker}{\mathrm{Ker}}

\newcommand{\Id}{\mathrm{Id}}
\newcommand{\Tr}{\mathrm{Tr}}
\newcommand{\new}{\mathrm{new}}
\newcommand{\old}{\mathrm{old}}

\newcommand{\lra}{\longrightarrow}
\newcommand{\ra}{\rightarrow}

\newcommand{\mrm}[1]{\mathrm{#1}}

\def\1{1\!\!1}

\newcommand{\psmat}[4]{\bigl( \begin{smallmatrix} #1 & #2 \\ #3 & #4 \end{smallmatrix} \bigr)}

\newcommand{\Spa}[1]{S_{k,l}(\Gamma_0(#1))}
\newcommand{\Tro}[2]{\mathrm{Tr}^{#1}_{#2}} 
\newcommand{\TroA}[2]{\mathrm{Tr^\prime}^{#1}_{#2}} 
\newcommand{\bstwomat}[2]{\bigl( \begin{smallmatrix} #1 \\ #2 \end{smallmatrix} \bigr)}
    
\def\dis{\displaystyle}

\title[Notes on Atkin-Lehner theory for Drinfeld modular forms]{Notes on Atkin-Lehner theory for Drinfeld modular forms}
\author[T. Dalal]{Tarun Dalal}
\email{ma17resch11005@iith.ac.in}
\address{
Department of Mathematics \\
Indian Institute of Technology Hyderabad\\
Kandi, Sangareddy - 502285\\
INDIA. 
}
\author[N. Kumar]{Narasimha Kumar}
\email{narasimha@math.iith.ac.in}
\address{
Department of Mathematics \\
Indian Institute of Technology Hyderabad\\
Kandi, Sangareddy - 502285\\
INDIA. 
}

\date{\today}
\begin{document}
\begin{abstract}
 In this article, we settle a part of the Conjecture by Bandini and Valentino (\cite{BV19a}) for $S_{k,l}(\Gamma_0(T))$ when $\mrm{dim}\ S_{k,l}(\GL_2(A))\leq 2$.
Then, we frame this conjecture for prime, higher levels, and provide some evidence in favor of it. 
For any square-free level $\n$, we define oldforms $S_{k,l}^{\old}(\Gamma_0(\n))$, newforms $S_{k,l}^{\new}(\Gamma_0(\n))$, and investigate their properties. These properties depend on the commutativity of the (partial) Atkin-Lehner operators  with the $U_\p$-operators.  Finally, we show that the set of all $U_\p$-operators are simultaneously  diagonalizable on $S_{k,l}^{\new}(\Gamma_0(\n))$.
\end{abstract}
\maketitle
\section{Introduction}
\label{Introduction}
The theory of oldforms and newforms is a well-understood area in the theory of classical modular forms. 
Certain properties of modular forms heavily depend on whether they belong to oldforms or newforms. 
For example, the space of newforms  has a basis consisting of normalized eigenforms for all the Hecke operators. In fact, the Fourier coefficients of these normalized eigenforms generate a number field. To the best of author's knowledge,    the analogues theory of oldforms and newforms is not much known for Drinfeld modular forms.

In this article, we propose a definition of oldforms,  newforms for Drinfeld modular forms of square-free level. We justify these definitions  by showing that these spaces are invariant under the action of the Hecke operators.  The proof requires the commutativity of the (partial) Atkin-Lehner operators  with the $U_\p$-operators and  certain properties of the space of $\p$-oldforms and $\p$-newforms. 

In a series of papers (cf.~\cite{BV19},~\cite{BV19a},~\cite{BV20},~\cite{Val22}), Bandini and Valentino have defined  $\p$-oldforms, $\p$-newforms and studied some of their properties. 
In~\cite{BV19}, the authors defined $T$-oldforms $S_{k,l}^{T-\old}(\Gamma_0(T))$,
$T$-newforms $S_{k,l}^{T-\new}(\Gamma_0(T))$ for $\p = (T)$. 
In \cite{BV19a}, a sequel to~\cite{BV19}, they have made the following conjecture:
\begin{conj}(\cite[Conjecture 1.1]{BV19a} for $\Gamma_0(T)$) 
\label{level T conjecture}
\begin{enumerate}[(i)]
\item $\ker(T_T)=0$ where $T_T$ is acting on $S_{k,l}(\GL_2(A)),$
\item $S_{k,l}(\Gamma_0(T))=S_{k,l}^{T-\old}(\Gamma_0(T))\oplus S_{k,l}^{T-\new}(\Gamma_0(T))$,
\item $U_T$ is diagonalizable on $S_{k,l}(\Gamma_0(T))$.
\end{enumerate}
\end{conj}
In~\cite{BV19a},~\cite{BV22}, the authors proved Conjecture~\ref{level T conjecture} in some special cases, using harmonic cocycles, 
the trace maps $\mrm{Tr}$ and $\mrm{Tr}^\prime$, and
the linear algebra interpretation of the Hecke operators $T_\p$ 
and $U_\p$-operators.
In this article, by studying the action of the $T_T$-operators on the 
Fourier coefficients of Drinfeld modular forms, we prove:
\begin{thm}[Theorem~\ref{Conjecture of BV for dimension 1}, Theorem~\ref{Direct sum decomposition dim 2 case}]
\label{Conjecture of BV for dimension 1 introduction}
\label{Direct sum decomposition dim 2 case introduction}
If $\dim S_{k,l}(\GL_2(A))\leq 1$, then
Conjecture~\ref{level T conjecture} is true for $S_{k,l}(\Gamma_0(T))$.
Furthermore,
if $\dim S_{k,l}(\GL_2(A))=2$,
then $S_{k,l}(\Gamma_0(T))=S_{k,l}^{T-\old}(\Gamma_0(T))\oplus S_{k,l}^{T-\new}(\Gamma_0(T))$
holds.
\end{thm} 
Our methods in the proof of Theorem~\ref{Conjecture of BV for dimension 1} and Theorem~\ref{Direct sum decomposition dim 2 case} are completely different from that of~\cite{BV19a},~\cite{BV22}. Our methods are based on the analysis of the Fourier coefficients of the image of an element via the Hecke operator $T_T$. We are very optimistic that our methods are suitable for generalizations, i.e., when $\dim S_{k,l}(\GL_2(A)) \geq 3$.

In a continuation work (\cite{BV20}), for any prime ideal $\p$, the authors extended the definition of $\p$-oldforms, $\p$-newforms to level $\p$, level $\p\m$ with $\p\nmid \m$. So, it is quite natural  to understand
Conjecture~\ref{level T conjecture} for level $\p$, level $\p\m$ with $\p\nmid \m$. In this article, we frame it as a  question (cf.~Question~\ref{example of direct sum at level P}) and provide some evidences in favor of it. 

First, we generalize the results of~ \cite{BV19a} for $\p=(T)$ to any arbitrary prime ideal $\p$ (cf. Proposition~\ref{exple1}). This implies that Question~\ref{example of direct sum at level P} has an affirmative answer in these cases. Then,  we exhibit some cases where Question~\ref{example of direct sum at level P}  for the level $\p\m$  is true (cf.  Proposition~\ref{exple2}). Here, we bring a word of caution. If $\m\ne A$, we show that the direct sum decomposition in Question~\ref{example of direct sum at level P}(2) may fail when $l=1$ (cf. Proposition~\ref{counter example for direct sum decomposition for square free p-part},  Remark~\ref{counter example higher powers remark}). More precisely, we exhibit non-zero Drinfeld cusp forms which are both $\p$-oldforms and $\p$-newforms. We believe that this is the only case where it may fail, and in fact, it may serve as a guiding example in future works.

In the final section, we propose a definition of oldforms  $S_{k,l}^{\old} (\Gamma_0(\n))$ and newforms $S_{k,l}^{\new} (\Gamma_0(\n))$ 
for Drinfeld modular forms of square-free level $\n$. In fact, we justify our definition by showing that these spaces are invariant under the action of the Hecke operators (cf.~Theorem \ref{Oldspace and Newspace are stable under Hecke operators}). This requires the commutativity of the (partial) Atkin-Lehner operators with the $T_\p$ and $U_\p$-operators. For the $T_\p$-operators, this is exactly Theorem 1.1 of~\cite{Val22}.
Then, we prove an analogous result for the $U_\p$-operators (cf. Theorem~\ref{U_P and A-L commutes}).
Finally, we prove that the $U_\p$-operators  are simultaneously diagonalizable on $S_{k,l}^{\new} (\Gamma_0(\n))$ (cf. Corollary~\ref{Basis consisting of eigenforms of U_P operators}). 

\subsection*{Notations:}
Throughout the article, we use the following notations:
\begin{itemize}
\item Let $p$ be an odd prime number and $q=p^r$ for some $r \in \N$.
\item Let $k\in \N$ and $l\in \Z/(q-1)\Z$ such that $k\equiv 2l \pmod {q-1}$.
      Let $0 \leq l \leq q-2$ be a lift of $l \in \Z/(q-1)\Z$. By abuse of notation,
      we  write $l$ for the integer as well as its class. 
\end{itemize}
Let $\F_q$ denote the finite field of order $q$. Set $A :=\F_q[T]$,  
$K :=\F_q(T)$. 
Let $K_\infty=\F_q((\frac{1}{T}))$ be the completion of $K$ 
with respect to the infinite place $\infty$ (corresponding to $\frac{1}{T}$-adic valuation), and denote by $C:= \widehat{\overline{K_\infty}}$, the completion of $\overline{K_{\infty}}$.
Let $\mfp = (P)$ denote a prime ideal of $A$ with a monic irreducible
polynomial $P$.

\subsection*{An overview of the article}
The article is organized as follows. In \S\ref{Basic theory}, we recall some basic theory of Drinfeld modular forms.  In \S\ref{Certain important operators section}, we introduce certain important operators and study the inter-relations between them. In \S\ref{For an arbitrary level n with p| n } we prove 
Theorem \ref{Conjecture of BV for dimension 1 introduction}
and study the validity of Question~\ref{example of direct sum at level P} for prime,  higher levels. In \S\ref{Oldforms and Newforms for square free level n}, we define oldforms, newforms and show that they are invariant under the action of the  Hecke operators. 

\section*{Acknowledgments}  
The authors would like to thank Professor Andrea Bandini for his encouragement to check the validity of Conjecture~\ref{level T conjecture}, when $\mrm{dim}\ S_{k,l}(\GL_2(A))=2$. We would also like to thank Professor Francesc Bars for his valuable comments on an earlier version of this article.  The first author thanks University Grants Commission (UGC), India for the financial support provided in the form of Research Fellowship to carry out this research work at IIT Hyderabad. 


\section{Basic theory of Drinfeld modular forms}
\label{Basic theory}
In this section, we recall some basic theory of Drinfeld modular forms
(cf. \cite{Gos80}, \cite{Gos80a},  \cite{Gek88},  \cite{GR96} for more details).




Let $L=\tilde{\pi}A \subseteq C$ be the $A$-lattice of rank $1$ 
corresponding to the rank $1$ Drinfeld module (which is also called Carlitz module) 
$\rho_T=TX+X^q$,
where $\tilde{\pi}\in K_\infty(\sqrt[q-1]{-T})$ is defined up to a $(q-1)$-th root of unity.
The Drinfeld upper half-plane $\Omega= C-K_\infty$, which is analogue to the complex upper half-plane, has a rigid analytic structure. 
The group $\GL_2(K_\infty)$ acts on $\Omega$ via fractional linear transformations.

\begin{dfn}
	Let $k\in \N$, $l \in \Z/(q-1)\Z$ and $f:\Omega \ra C$ 
	be a rigid holomorphic function on $\Omega$. For any $\gamma=\psmat{a}{b}{c}{d}\in \GL_2(K_{\infty})$,
	we define the slash operator $|_{k,l} \gamma$ on $f$ by
	\begin{equation}\label{slash operator}
	f|_{k,l} \gamma := (\det \gamma)^{l}(cz+d)^{-k}f(\gamma z).
	\end{equation}
\end{dfn}
For an ideal $\mfn \subseteq A$, let $\Gamma_0(\mfn)$ denote the congruence subgroup 
$\{\psmat{a}{b}{c}{d}\in \GL_2(A): c\in \mfn \}$.
Now, we define a Drinfeld modular form of weight $k$, type $l$ for $\Gamma_0(\mfn)$:
\begin{dfn}
	\label{Definition of DMF}
	A rigid holomorphic function $f:\Omega \ra C$ is said to be a Drinfeld modular form of weight $k$, type $l$ 
	for $\Gamma_0(\mfn)$ if 
	\begin{enumerate}
		\item $f|_{k,l}\gamma= f$ , $\forall \gamma\in \Gamma_0(\mfn)$,
		\item $f$ is holomorphic at the cusps of $\Gamma_0(\mfn)$.
	\end{enumerate}
	The space of Drinfeld modular forms of weight $k$, type $l$ for $\Gamma_0(\mfn)$ is denoted by $M_{k,l}(\Gamma_0(\mfn)).$
	Furthermore, if $f$ vanishes at the cusps of $\Gamma_0(\mfn)$, then we say $f$ is a Drinfeld cusp form of
	weight $k$, type $l$ for $\Gamma_0(\mfn)$ and the space of such forms is denoted by $S_{k,l}(\Gamma_0(\mfn))$.
\end{dfn}

If $k\not \equiv 2l \pmod {q-1}$, then $M_{k,l}(\Gamma_0(\mfn))=\{0\}$. So, without loss of generality, we can assume that $k\equiv 2l \pmod {q-1}$.  Let $u(z) :=  \frac{1}{e_L(\tilde{\pi}z)}$, where $e_L(z):= z{\prod_{\substack{0 \ne \lambda \in L }}}(1-\frac{z}{\lambda}) $
be the exponential function attached to the lattice $L$. 
Then, each Drinfeld modular form  $f\in M_{k,l}(\Gamma_0(\mfn))$ has a unique $u$-series expansion at $\infty$
given by $f=\sum_{i=0}^\infty a_f(i)u^i$.
Since $\psmat{\zeta}{0}{0}{1} \in \Gamma_0(\n)$ for $\zeta \in \F_q^\times$, condition $(1)$ of Definition~\ref{Definition of DMF} 
implies $a_f(i)=0$ if $i\not\equiv l \pmod {q-1}$.
Hence, the $u$-series expansion of $f$ at $\infty$ can be written as
$\sum_{0 \leq \ i \equiv l \mod (q-1)} a_f(i)u^{i}.$ 
Any Drinfeld modular form of type $l \neq 0$ is a cusp form, i. e., $M_{k,l}(\Gamma_0(\n))=S_{k,l}(\Gamma_0(\n))$.

\subsection{Examples}
We now give some examples of Drinfeld modular forms. 
\begin{exple}[\cite{Gos80}, \cite{Gek88}]
	\label{Eisenstein Series}
	Let $d\in \N$. For $z\in \Omega$, the function
	\begin{equation*}
	g_d(z) := (-1)^{d+1}\tilde{\pi}^{1-q^d}L_d \sum_{\substack{a,b\in \F_q[T] \\ (a,b)\ne (0,0)}} \frac{1}{(az+b)^{q^d-1}}
	\end{equation*}
	is a Drinfeld modular form of weight $q^d-1$, type $0$ for $\GL_2(A)$,
	where $\tilde{\pi}$ is the Carlitz period and $L_d:=(T^q-T)\ldots(T^{q^d}-T)$ is the least common multiple of all monics of degree $d$.
	We refer $g_d$ as an Eisenstein series and it does not vanish at $\infty$.
\end{exple}

\begin{exple}[\cite{Gos80a}, \cite{Gek88}]
	\label{Delta-function}
	For $z\in \Omega$, the function
	\begin{equation*}
	\Delta(z) := (T-T^{q^2})\tilde{\pi}^{1-q^2}E_{q^2-1} + (T^q-T)^q\tilde{\pi}^{1-q^2}(E_{q-1})^{q+1},
	\end{equation*}
	is a Drinfeld cusp form of weight $q^2-1$,  type $0$ for $\GL_2(A)$,
	where $E_k(z)= \sum_{\substack{(0,0)\ne (a,b)\in A^2}}\frac{1}{(az+b)^k}.$
\end{exple}
\begin{exple}[\cite{Gek88}]
	\label{Poincare Series}
	For $z \in \Omega$, the function
	\begin{equation*}
	h(z) := \sum_{\gamma = \psmat{a}{b}{c}{d} \in H\char`\\ \GL_2(A)} \frac{\det \gamma}{(cz+d)^{q+1}}  u(\gamma z),
	\end{equation*}
	where $H=\big\{\psmat{*}{*}{0}{1}\in \GL_2(A)\big\}$,
is a Drinfeld cusp form of weight $q+1$, type $1$ for $\GL_2(A)$.
\end{exple}

We end this section by introducing an important function $E$, which is not modular.
In~\cite{Gek88}, Gekeler defined the function $E(z):= \frac{1}{\tilde{\pi}} \sum_{\substack{a\in \F_q[T] \\ a \ \mathrm{monic}}} ( \sum_{b\in \F_q[T]} \frac{a}{az+b}),$ which is analogous to the Eisenstein series of weight $2$ over $\Q$. For any prime ideal $\p=(P)$, using $E$, we can construct the following Drinfeld modular form
\begin{equation}
\label{E_T}
E_P(z) := E(z)- PE(Pz) \in S_{2,1}(\Gamma_0(\p)).
\end{equation} 
(cf. \cite[Proposition 3.3]{DK1} for a detailed discussion about $E_P$).

\section{Certain important operators}
\label{Certain important operators section}
In this section, we recall certain important operators and study their properties.
\subsection{Atkin-Lehner operators:}
Let $\mfr$, $\n$ be two ideals of $A$ generated by monic polynomials $r$, $n$, respectively, with $\mfr \mid \n$. The following definition can be found in~\cite[Page 331]{Sch96}.
\begin{dfn}
For $\mfr||\n$ (i.e., $\mfr|\n$ with $(\mfr, \frac{\n}{\mfr})=1$), 
the (partial) Atkin-Lehner operator $W_\mfr^{(\n)}$ is defined by the action of the matrix 
$\psmat {ar}{b}{cn}{dr}$ on $M_{k,l}(\Gamma_0(\n))$, where $a,b, c,d\in A$ such that $adr^2-bcn=\zeta\cdot  r$
for some $\zeta\in \F_q^\times$.
\end{dfn}
By~\cite[Proposition 3.2]{DK1}, the action of $W_\mfr^{(\n)}$ on $M_{k,l}(\Gamma_0(\n))$ is well-defined  (here the action of the slash operator is  different from the one in~\cite{DK1}).
Assume that $\p^\alpha||\n$ with $\alpha\in \N$. 
We now fix some representatives for the (partial) Atkin-Lehner operators $W_{\p^\alpha}^{(\n)}$ and $W_{\p^{\alpha-1}}^{(\m)}$.
    
\begin{dfn}
For  $f\in S_{k,l}(\Gamma_0(\n))$, we write
$f|_{k,l}{W_{\p^\alpha}^{(\n)}} := f|_{k,l} \psmat{P^\alpha}{b}{n}{P^\alpha d},$
where  $b,d\in A$ such that  $P^{2\alpha}d-nb= P^\alpha.$ 
Since $(P^\alpha,\frac{n}{P^\alpha})=1$, such $b,d \in A$ exist.

\end{dfn}

Write $n=Pm$ and $\m=(m)$.
When $\alpha \geq 2$, we can take $\psmat{P^{\alpha-1}}{b}{m}{P^\alpha d}$
as a representative for the (partial) Atkin-Lehner operator $W_{\p^{\alpha-1}}^{(\m)}$. 
\begin{lem}
\label{Atkin-Lehner operator is an involution}
The operator $|_{k,l} W_{\p^\alpha}^{(\n)}$  on $S_{k,l}(\Gamma_0(\n))$ defines an endomorphism
and  for all $f\in S_{k,l}(\Gamma_0(\n))$, we have  $(f|_{k,l}{W_{\p^\alpha}^{(\n)}})|_{k,l}{W_{\p^\alpha}^{(\n)}} = P^{\alpha(2l-k)}f$

\end{lem}
\begin{proof}
Since $W_{\p^\alpha}^{(\n)} \cdot W_{\p^\alpha}^{(\n)} = \psmat{P^{\alpha}}{0}{0}{P^\alpha}\gamma$
for some $\gamma\in \Gamma_0(\n)$, the lemma follows.
\end{proof}

\begin{lem}
\label{A-L commutes among themselves}
For $i=1,2$, let $\p_i$ be two distinct prime ideals of $A$
such that $\p_i^{\alpha_i}||\n$ for some $\alpha_i\in \N.$
Then $W_{\p_1^{\alpha_1}}^{(\n)}W_{\p_2^{\alpha_2}}^{(\n)}=W_{\p_2^{\alpha_2}}^{(\n)}W_{\p_1^{\alpha_1}}^{(\n)}.$
\end{lem}

\begin{proof}
The proof of the lemma follows from $W_{\p_1^{\alpha_1}}^{(\n)}W_{\p_2^{\alpha_2}}^{(\n)}=W_{\p_1^{\alpha_1}\p_2^{\alpha_2}}^{(\n)}=W_{\p_2^{\alpha_2}}^{(\n)}W_{\p_1^{\alpha_1}}^{(\n)}$.
\end{proof}

\subsection{Hecke operators:}
We now recall the definitions of $T_\p$ and $U_\p$-operators. 

\begin{dfn}
\label{Definitions of T_p and U_p operators}
For $f \in S_{k,l}(\Gamma_0(\n))$, we define
\begin{align*}
T_\p(f) &:= P^{k-l} \sum_{\substack{Q\in A\\ \deg Q<\deg P}}f|_{k,l}\psmat{1}{Q}{0}{P} + P^{k-l}f|_{k,l}\psmat{P}{0}{0}{1} \quad \mathrm{if} \ \p\nmid \n,\\
U_\p(f) &:= P^{k-l} \sum_{\substack{Q\in A\\ \deg Q < \deg P}}f|_{k,l}\psmat{1}{Q}{0}{P} \quad \mathrm{if} \ \p\mid \n.
\end{align*}
\end{dfn}
The commutativity of the $T_\mfp$ and $U_\p$-operators  is content of the next proposition.
\begin{prop}
\label{U_p, U_Q commutes and U_P, T_Q commutes}
Let $\n$ be an ideal of $A$ and $\mfp_1,\mfp_2$ be two distinct prime ideals of $A$ generated by  monic irreducible polynomials $P_1,P_2$ respectively. Suppose that $\p_1\mid \n$.
Then,  $U_{\p_1}$ commutes with $U_{\p_2}$ (resp., with $T_{\p_2}$) if ${\p_2} \mid \n$ (resp., if ${\p_2} \nmid \n$)
as operators on $S_{k,l}(\Gamma_0(\n))$.
\end{prop}

\begin{proof}
Since $P_1$ and $P_2$ are distinct primes, for any $b\in A$ with $\deg b<\deg P_1$ there exists a unique $b^\prime\in A$ with $\deg b^\prime < \deg P_1$ such that $P_1|(b- b^\prime P_2)$. Thus, $\psmat{1}{\frac{b-b^\prime P_2}{P_1}}{0}{1}\in \Gamma_0(\n)$ and $\psmat{1}{b}{0}{P_1}\psmat{P_2}{0}{0}{1} = \psmat{1}{\frac{b-b^\prime P_2}{P_1}}{0}{1}  \psmat{P_2}{0}{0}{1}\psmat{1}{b^\prime}{0}{P_1}.$ 
Now the result follows from Definition \ref{Definitions of T_p and U_p operators} and the following equality 
$\small \dis \sum_{\substack{b\in A\\ \deg b<\deg P_1}}\sum_{\substack{d\in A\\ \deg d<\deg P_2}} \psmat{1}{b+dP_1}{0}{P_1P_2}=\sum_{\substack{c\in A \\ \deg c<\deg P_1 + \deg P_2}} \psmat{1}{c}{0}{P_1P_2}=\sum_{\substack{d^\prime\in A\\ \deg d^\prime<\deg P_2}} \sum_{\substack{b^\prime\in A\\ \deg b^\prime<\deg P_1}} \psmat{1}{d^\prime+b^\prime P_2}{0}{P_1P_2}.$
\end{proof}

\subsection{The Trace operators}
We define the trace operators and mention some of its properties.
\begin{dfn}
For any ideal $\mfr|\n$, we define the trace operator
$\Tr_{\mfr}^\n : M_{k,l}(\Gamma_0(\n)) \longrightarrow M_{k,l}(\Gamma_0(\mfr))$
by
$\Tr_{\mfr}^\n(f) = \sum_{\gamma\in \Gamma_0(\n)\char`\\ \Gamma_0(\mfr)} f|_{k,l}\gamma. $
\end{dfn}
We conclude this section with a proposition where we explicitly compute the action of the trace operator in terms of the (partial) Atkin-Lehner operators and the Hecke operators.
\begin{prop}
\label{relation between Trace, A-L and U_P operator for level pm}
\label{relation between Trace Atkin_Lehner and U_p operators}
Let $\mfm,\mfn$ be two ideals of $A$ generated by monic polynomials $m, n$, respectively,
such that $n=Pm$. Let $\alpha \in \N$ such that $P^\alpha||n$. If $f\in S_{k,l}(\Gamma_0(\n)),$ then
$$\Tro{\n}{\m}(f)= 
\begin{cases}
f+ P^{-l}U_\p(f|W_\p^{(\n)})  \qquad \qquad  \qquad \quad \ \quad \mathrm{if} \ \alpha=1,\\
P^{-l-(\alpha-1)(2l-k)}U_\p(f|W^{(\n)}_{\p^\alpha})|_{k,l}W^{(\m)}_{\p^{\alpha-1}} \quad \mathrm{if} \ \alpha\geq 2.
\end{cases}
$$
\end{prop}
\begin{proof}
If $\alpha=1$, then this proposition is exactly~\cite[Proposition $3.6$]{DK1}. When $\n$ is a prime ideal, this coincides with~\cite[Proposition 3.8]{Vin14}. Note that, the action of the slash operator here is different from there.

Now, we let  $\alpha\geq 2$. 
By definition, we have
\begin{align*}
U_P(f|W^{(\n)}_{\p^\alpha}) &= P^{k-l} \sum\limits_{\deg Q<\deg P}f|_{k,l}\psmat{P^\alpha}{b}{n}{P^\alpha d}\psmat{1}{Q}{0}{P}
 = P^{k-l} \sum\limits_{\deg Q<\deg P}f|_{k,l} \psmat{P^\alpha}{P^\alpha Q+bP}{n}{nQ+P^{\alpha+1}d}\\
&= P^{k-l} \sum\limits_{\deg Q<\deg P}f|_{k,l}\psmat{P}{0}{0}{P}\psmat{P^{\alpha-1}}{P^{\alpha-1} Q+b}{m}{mQ+P^{\alpha}d}
= P^{l} \sum\limits_{\deg Q<\deg P}f|_{k,l}\psmat{P^{\alpha-1}}{P^{\alpha-1} Q+b}{m}{mQ+P^{\alpha}d}\\
&= P^{l} \sum\limits_{\deg Q<\deg P}f|_{k,l} \psmat{1-mQ}{P^{\alpha-1} Q}{-\frac{m^2}{P^{\alpha-1}}Q}{1+mQ} \psmat{P^{\alpha-1}}{b}{m}{P^\alpha d}.
\end{align*}
We now show that $\big\{\psmat{1-mQ}{P^{\alpha-1} Q}{-\frac{m^2}{P^{\alpha-1}}Q}{1+mQ} : \deg Q<\deg P\big\}$ is a set of representatives for  $\Gamma_0(\n)\char`\\ \Gamma_0(\m)$. Let $\psmat{s}{t}{mx}{y}\in \Gamma_0(\m),$ where $s,t,x,y \in A$ satisfy $sy-tmx=\zeta \in \F_q^\times$.
Let $-\zeta^{-1}sx \equiv Q_1 \pmod P$, where $Q_1\in A$ such that $\deg Q_1< \deg P.$ 
Since $P^{\alpha-1}||m$, there exists an unique $Q_2\in A$ with $\deg Q_2< \deg P$ such that $\frac{m}{P^{\alpha-1}}Q_2\equiv 1 \pmod P.$ 
Since $P|m$, the choice of $Q_1$ and $sy-tmx=\zeta \in \F_q^\times$
implies that $x+yQ_1\equiv 0 \pmod P.$
Let $Q\in A$ with $\deg Q< \deg P$ such that $Q_1Q_2\equiv Q\pmod P$. 
Then $x+\frac{m}{P^{\alpha-1}}Qy \equiv 0 \pmod P$.
Hence, we get $\psmat{s(1+mQ)+t\frac{m^2}{P^{\alpha-1}}Q}{t(1-mQ)-sP^{\alpha-1}Q}{mx(1+mQ)+y\frac{m^2}{P^{\alpha-1}}Q}{y(1-mQ)-mxP^{\alpha-1}Q}\in \Gamma_0(\n)$ and 
$$\psmat{s}{t}{mx}{y}=\psmat{s(1+mQ)+t\frac{m^2}{P^{\alpha-1}}Q}{t(1-mQ)-sP^{\alpha-1}Q}{mx(1+mQ)+y\frac{m^2}{P^{\alpha-1}}Q}{y(1-mQ)-mxP^{\alpha-1}Q}\psmat{1-mQ}{P^{\alpha-1} Q}{-\frac{m^2}{P^{\alpha-1}}Q}{1+mQ}.$$
Thus, the set 
$\big\{\psmat{1-mQ}{P^{\alpha-1} Q}{-\frac{m^2}{P^{\alpha-1}}Q}{1+mQ} : \deg Q<\deg P\big\}$
forms a complete set of representatives for $\Gamma_0(\n)\char`\\ \Gamma_0(\m)$. Therefore 
\begin{equation*}
U_P(f|W^{(\n)}_{\p^\alpha})=P^{l} \sum\limits_{\deg Q<\deg P}f|_{k,l} \psmat{1-mQ}{P^{\alpha-1} Q}{-\frac{m^2}{P^{\alpha-1}}Q}{1+mQ} \psmat{P^{\alpha-1}}{b}{m}{P^\alpha d} = P^{l}(\Tr^\n_\m f)|_{k,l}W^{(\m)}_{\p^{\alpha-1}}.
\end{equation*}
Applying $W^{(\m)}_{\p^{\alpha-1}}$ operator on both sides, the proposition
follows from Lemma~\ref{Atkin-Lehner operator is an involution}.
\end{proof}

\begin{cor}
\label{Tr of pm to m of level m}
Let $\p,\m$ be with $(\p,\m)=1$. If $f\in S_{k,l}(\Gamma_0(\p))$, then
$\Tro{\p\m}{\m}(f)= \Tro{\p}{1}(f).$
\end{cor}
\begin{proof}
Since $f|W^{(\p\m)}_\p=f|W^{(\p)}_\p$, the result follows from Proposition~\ref{relation between Trace Atkin_Lehner and U_p operators}.
\end{proof}

\section{$\p$-oldforms and $\p$-newforms for level $\p\m$}
\label{For an arbitrary level n with p| n }
%
%


Let $\p$ be a prime ideal of $A$. Throughout this section, we consider $\m$  an ideal of $A$ generated by a monic polynomial $m$ such that $\p\nmid \m$. We first recall the definitions of $\p$-oldforms and $\p$-newforms (cf.~\cite{BV20},~\cite{Val22}). 
Consider the map
\begin{align*}
(\delta_1, \delta_P): (S_{k,l}(\Gamma_0(\m)))^2  \lra S_{k,l}(\Gamma_0(\p\m))\ \mathrm{defined \ by} \ 
(f,g)  \lra \delta_1f + \delta_P g,
\end{align*} 
where 
$\delta_1$, $\delta_P : S_{k,l}(\Gamma_0(\m)) \ra S_{k,l}(\Gamma_0(\p\m))$ given by 
$\delta_1(f)=f$ and $\delta_P(f)= f|_{k,l}\psmat{P}{0}{0}{1}.$

\begin{dfn}
The space of $\p$-oldforms $S_{k,l}^{\p-\old}(\Gamma_0(\p\m))$ of level $\p\m$ is defined 
as the subspace of $S_{k,l}(\Gamma_0(\p\m))$ generated by the image of $(\delta_1, \delta_P)$.
\end{dfn}

\begin{dfn}
The space of $\p$-newforms $S_{k,l}^{\p-\new}(\Gamma_0(\p\m))$ of level $\p\m$ is defined as 
$$S_{k,l}^{\p-\new}(\Gamma_0(\p\m)):= \Ker(\Tro{\p\m}{\m})\cap \Ker(\TroA{\p\m}{\m}), \quad \mathrm{where} \ \TroA{\p\m}{\m} f : = \Tro{\p\m}{\m}(f|W_\p^{(\p\m)}).$$
\end{dfn}

We wish to understand Conjecture~\ref{level T conjecture} for prime $\p$, higher levels $\p\m$. We now formulate it as a question and provide some evidences in favor of it. More precisely:
\begin{ques}[For level $\p\m$]
Suppose  $\m$ is an ideal of $A$ such that $\p \nmid \m.$ 
\label{example of direct sum at level P}
 
\begin{enumerate}
\item 
 $\ker(T_\p)=0$, where $T_\p  \in \mrm{End}(S_{k,l}(\Gamma_0(\m))),$
\item 
\begin{equation} 
  \label{direct sum decomposition level pm in corollary}
  S_{k,l}(\Gamma_0(\p\m))=S_{k,l}^{\p-\old}(\Gamma_0(\p\m))\oplus S_{k,l}^{\p-\new}(\Gamma_0(\p\m)),      
\end{equation}
\item 
The $U_\p$-operator is diagonalizable on $S_{k,l}(\Gamma_0(\p\m)).$
\end{enumerate}
\end{ques}
When we say that ``Question~\ref{example of direct sum at level P} is true for level $\p\m$", we mean  all the statements of Question~\ref{example of direct sum at level P} are true. We first show, if $\m=A, \p = (P)$ with $\deg P=1$, then Question~\ref{example of direct sum at level P} is true for level $\p$ if $\dim \ S_{k,l}(\GL_2(A))\leq 1$.  In particular, Conjecture~\ref{level T conjecture} is true for $S_{k,l}(\Gamma_0(T))$. Furthermore, we show the direct sum decomposition in Question~\ref{example of direct sum at level P}(2) holds for $S_{k,l}(\Gamma_0(\p))$ if $\dim S_{k,l}(\GL_2(A))\leq 2$. Finally, we give some evidences in the support of Question~\ref{example of direct sum at level P} for level $\p\m$.

\subsection{Question~\ref{example of direct sum at level P} when $\mathbf{dim}\ S_{k,l}(\GL_2(A))\leq 2$:}
\label{Discussion on (2)}
We now discuss on  implications of Question~\ref{example of direct sum at level P}(1),(2) to Question~\ref{example of direct sum at level P}(3).
If Question~\ref{example of direct sum at level P}(2) is true, then  the diagonalizability of the $U_\p$-operator on $S_{k,l}(\Gamma_0(\p\m))$ depends on that of the $U_\p$-operators on $S_{k,l}^{\p-\old}(\Gamma_0(\p\m))$,
$S_{k,l}^{\p-\new}(\Gamma_0(\p\m)).$
By~\cite[Remark 2.17]{BV20}, the $U_\p$-operator is diagonalizable on $S_{k,l}^{\p-\new}(\Gamma_0(\p\m))$. However, the $U_\p$-operator   is diagonalizable on $S_{k,l}^{\p-\old}(\Gamma_0(\p\m))$
if and only if the $T_{\p}$-operator is diagonalizable on $S_{k,l}(\Gamma_0(\m))$ and is injective (cf.~\cite[Remark 2.4]{BV20}). Therefore, if Question~\ref{example of direct sum at level P}(1),(2) are true, then
Question~\ref{example of direct sum at level P}(3) is equivalent to check the diagonalizability of the $T_\p$-operator on $S_{k,l}(\Gamma_0(\m)).$

\subsubsection{Reformulation of Question~\ref{example of direct sum at level P}(2)}
In~\cite{Val22}, Valentino gave a necessary and sufficient condition for Question~\ref{example of direct sum at level P}(2) to hold. More precisely: 
 \begin{thm}\cite[Theorem 3.15]{Val22}
\label{direct sum decomposition as in Valentino}
The map $\Id- P^{k-2l}(\TroA{\p\m}{\m})^2$ is bijective on $S_{k,l}(\Gamma_0(\p\m))$ if and only if 
Question~\ref{example of direct sum at level P}(2) holds.
\end{thm}

We  now rephrase  Theorem~\ref{direct sum decomposition as in Valentino} in terms of the eigenvalues of the $T_{\p}$-operator.
\begin{prop}
\label{Direct sum pm case}
The $T_\p$-operator has no eigenform on $\Spa{\m}$ with eigenvalues $\pm P^{\frac{k}{2}}$ if and only if Question~\ref{example of direct sum at level P}(2) holds.
\end{prop}
The proof of Proposition~\ref{Direct sum pm case} depends on the following observations. For any $\varphi\in S_{k,l}(\Gamma_0(\m))$, we have:
\begin{equation}
\label{delta_P and A-L together}
\begin{aligned}
\varphi|_{k,l} W_\p^{(\p\m)} &= \varphi|_{k,l}\psmat{1}{b}{m}{dP}\psmat{P}{0}{0}{1} = \varphi|_{k,l}\psmat{P}{0}{0}{1}= \delta_P\varphi, \\
(\delta_P\varphi)|_{k,l} W_\p^{(\p\m)} &= \varphi|_{k,l}\psmat{P}{0}{0}{1}\psmat{P}{b}{Pm}{dP}= P^{2l-k}\varphi.
\end{aligned}
\end{equation}
Combining Proposition~\ref{relation between Trace Atkin_Lehner and U_p operators} with~\eqref{delta_P and A-L together} we obtain
\begin{equation}
\label{Trace and delta P}
\Tr^{\prime \p\m}_\m (\delta_1(\varphi))=\varphi |_{k,l}W_\p^{(\p\m)}+P^{l-k}(U_\p(\delta_1(\varphi)))=\delta_P \varphi+P^{l-k}(U_\p(\delta_1(\varphi)))= P^{l-k}T_\p\varphi,
\end{equation}
where $W_\p^{(\p\m)}:=\psmat{P}{b}{Pm}{d P}$ for some $b,d\in A$  with $dP^2-bPm=P.$

\begin{proof}[Proof of Proposition~\ref{Direct sum pm case}]
The proof of Theorem~\ref{direct sum decomposition as in Valentino} implies that if $f\in \ker(\Id - P^{k-2l}(\TroA{\p\m}{\m})^2)$, then $f\in \mrm{Im}(\delta_1)$. Thus $\ker(\Id - P^{k-2l}(\TroA{\p\m}{\m})^2) \subseteq S_{k,l}(\Gamma_0(\m)).$
 Therefore, $\Id - P^{k-2l}(\TroA{\p\m}{\m})^2$ is bijective on $S_{k,l}(\Gamma_0(\p\m))$ if and only if it is bijective on $S_{k,l}(\Gamma_0(\m))$.

For any $f\in S_{k,l}(\Gamma_0(\m))$,~\eqref{Trace and delta P} implies
$\TroA{\p\m}{\m}(\TroA{\p\m}{\m}(f))
= P^{l-k}\TroA{\p\m}{\m}(T_\p(f))= P^{2l-2k}T_\p(T_\p(f))$.
Thus, on $S_{k,l}(\Gamma_0(\m))$, we have $\Id - P^{k-2l}(\TroA{\p\m}{\m})^2 = \Id - P^{-k}T_\p^2$. 
Observe that the map $\Id - P^{-k}T_\p^2$ is bijective on $S_{k,l}(\Gamma_0(\m))$ if and only if the $T_\p$-operator has no eigenform on $\Spa{\m}$ with eigenvalues $\pm P^{\frac{k}{2}}$.
Now the proposition follows from Theorem~\ref{direct sum decomposition as in Valentino}.
\end{proof}

We now  prove a part of the Conjecture~\ref{level T conjecture} for $S_{k,l}(\Gamma_0(T))$ when  $\dim\ S_{k,l}(\GL_2(A))\leq 2$. 
We first prove that, Conjecture~\ref{level T conjecture} is true for $S_{k,l}(\Gamma_0(T))$ when  $\dim\ S_{k,l}(\GL_2(A))\leq 1$.  
\begin{thm}
\label{Conjecture of BV for dimension 1}
For $\m =A, \deg P =1$, Question~\ref{example of direct sum at level P} 
is true for $S_{k,l}(\Gamma_0(P))$ when $\dim S_{k,l}(\GL_2(A))\leq 1$.
In particular, Conjecture~\ref{level T conjecture} is true for 
$S_{k,l}(\Gamma_0(T))$ when $\dim S_{k,l}(\GL_2(A))\leq 1$.
\end{thm}
Theorem~\ref{Conjecture of BV for dimension 1} can be thought of as a continuation of the work done by Bandini and Valentino in~\cite{BV19a},~\cite{BV20}, and ~\cite{BV22}. Their method mainly involves  harmonic cocycles, the trace maps $\mrm{Tr}$ and $\mrm{Tr}^\prime$, and
the linear algebra interpretation of the Hecke operators $T_\p$ 
and $U_\p$-operators. However, we prove Theorem~\ref{Conjecture of BV for dimension 1} based on the analysis of the Fourier coefficients of the image of an element via the $T_\p$-operator, which may be suitable for generalizations. Finally, recall that $\dim M_{k,l}(\GL_2(A)) = \Big[\frac{k-l(q+1)}{q^2-1}\Big] +1$ (cf. \cite[Proposition 4.3]{Cor97}). 
By~\cite[Theorem 5.13]{Gek88}, the graded algebra $\oplus_{k,l}M_{k,l}(\GL_2(A))$ is generated by
$g_1,h$ .

\begin{proof}[Proof of Theorem~\ref{Conjecture of BV for dimension 1}]
Suppose $\dim S_{k,l}(\GL_2(A))=0$. Then Question~\ref{example of direct sum at level P}(1) is trivially true. Question~\ref{example of direct sum at level P}(2) and (3) are true by Proposition~\ref{Direct sum pm case}, by the diagonalizability of the $T_\p$-operator on $S_{k,l}(\GL_2(A))$.

Now, suppose  $\dim S_{k,l}(\GL_2(A))=1$. Clearly, the $T_\p$-operator is diagonalizable on $S_{k,l}(\GL_2(A)).$ Therefore, combining Proposition~\ref{Direct sum pm case} with the discussions in \S\ref{Discussion on (2)}, Question~\ref{example of direct sum at level P} 
for the level $(P) $ is true  if we show $\ker(T_\p)=0$ and the $T_\p$-operator has no eigenform on $S_{k,l}(\GL_2(A))$ with eigenvalues $\pm P^{\frac{k}{2}}$. We prove this statement in two cases, i.e.,  for  $l \neq 0$ and $l = 0$.

We first consider the case $l \neq 0$. 
In this case, $S_{k,l}(\GL_2(A))=\langle g_1^x h^l\rangle$ for some $x\in \{0, \ldots, q\}$ such that $k=x(q-1)+l(q+1)$.
 The $u$-series expansions of $g_1,  h$ are given by
$$g_1=1- (T^q-T)u^{q-1}- (T^q-T) u^{(q-1)(q^2-q+1)}+\cdots\in A[[u]],$$
$$h=-u-u^{1+(q-1)^2}+(T^q-T)u^{1+q(q-1)}-u^{1+(2q-2)(q-1)}+\cdots \in A[[u]].$$
Therefore, 
$g_1^xh^l= (-1)^l\sum_{i=0}^x (-1)^i\bstwomat{x}{i}(T^q-T)^iu^{i(q-1)+l} + O(u^{(q-1)^2+l})\in A[[u]].$
Let $T_\p(g_1^xh^l)=\sum_{j=0}^\infty a_{T_\p(g_1^xh^l)}(j(q-1)+l)u^{j(q-1)+l}$. 
By~\cite[Example 7.4]{Gek88}, we have
$$ a_{T_\p(g_1^xh^l)}(l)= \sum_{0\leq j < l} \bstwomat{l-1}{j}P^{l-j} a_{g_1^xh^l}(j(q-1)+l)\in A.$$
We define $x_0:=\min\{x,l-1\}$. Then,
$ a_{T_\p(g_1^xh^l)}(l)= \sum_{0\leq j \leq x_0} \bstwomat{l-1}{j} P^{l-j} (-1)^{l+j} \bstwomat{x}{j}(T^q-T)^j.$
Clearly, the set $\{0\leq j \leq x_0\mid  \bstwomat{l-1}{j}\bstwomat{x}{j} \neq 0\}$ is non-empty
and let $j_\mrm{max}$ be its maximum. 
Since $\deg (P^{l-j}(T^q-T)^j)<\deg (P^{l-(j+1)}(T^q-T)^{j+1})$, we get 
\begin{equation}
\label{degree bound for the lth coefficient of T_p operator on dimension 1}
0<\deg (a_{T_\p(g_1^xh^l)}(l))=  l+j_\mrm{max} (q-1)\leq l+x_0 (q-1) < \frac{x(q-1)+l(q+1)}{2}.
\end{equation}
The first inequality in~\eqref{degree bound for the lth coefficient of T_p operator on dimension 1} shows that 
$\ker(T_\p)=0.$
The inequality $\deg (a_{T_\p(g_1^xh^l)}(l))<\frac{x(q-1)+l(q+1)}{2}$
in~\eqref{degree bound for the lth coefficient of T_p operator on dimension 1}
shows that $T_\p(g_1^xh^l)$ cannot be equal to $\pm P^{\frac{x(q-1)+l(q+1)}{2}}g_1^xh^l$. 
In particular, the $T_\p$-operator has no eigenform on $S_{k,l}(\GL_2(A))$ 
with eigenvalues $\pm P^{\frac{k}{2}}$.

 

We now consider the case of $l=0$.
In this case, $S_{k,0}(\GL_2(A)) =\langle g_1^x\Delta \rangle$ for some $x\in \{0,\ldots, q\}$ such that $k=x(q-1)+(q^2-1)$. Recall that
$\Delta= -u^{q-1}+u^{q(q-1)}-(T^q-T)u^{(q+1)(q-1)} +O(u^{(q^2-q+1)(q-1)})\in A[[u]].$
Hence
$$g_1^x\Delta= \sum_{i=0}^x \bstwomat{x}{i}(-1)^{i+1}(T^q-T)^iu^{(i+1)(q-1)} + O(u^{q(q-1)})\in A[[u]].$$
 In this case, we consider the coefficient $a_{T_\p(g_1^x\Delta)}(q-1)$
to prove our claims. Since $a_{g_1^x\Delta}(0)=0$, we have
$a_{T_\p(g_1^x\Delta)}(q-1)= \sum_{0\leq j < q-1} \bstwomat{q-2}{j} P^{q-1-j} a_{g_1^x\Delta}((j+1)(q-1))$
(cf. \cite[Example 7.4]{Gek88}). We define $y_0 :=\min\{x, q-2\}$. Then, 
$$a_{T_\p(g_1^x\Delta)}(q-1)=\sum_{0\leq j \leq y_0} \bstwomat{q-2}{j}P^{q-1-j} (-1)^{j+1}  \bstwomat{x}{j}  (T^q-T)^j.$$
Arguing as in the previous case, i.e., $l \neq 0$, we get
$0<\deg (a_{T_\p(g_1^x\Delta)}(q-1))< \frac{x(q-1)+(q^2-1)}{2},$
which shows that  $\ker(T_\p)=0$ and the $T_\p$-operator has no eigenform on $S_{k,l}(\GL_2(A))$ with eigenvalues $\pm P^{\frac{k}{2}}$. This completes the proof of the proposition.
\end{proof}

We now show that, a part of the Conjecture~\ref{level T conjecture} is true for $S_{k,l}(\Gamma_0(T))$ when  $\dim\ S_{k,l}(\GL_2(A)) = 2$.  More precisely,
\begin{thm}
\label{Direct sum decomposition dim 2 case}
Let $\m=A$ and $\deg P=1$.
If $\dim S_{k,l}(\GL_2(A)) = 2$, then the direct sum decomposition in Question~\ref{example of direct sum at level P}(2) is true for $S_{k,l}(\Gamma_0(\p))$.
\end{thm}
\begin{proof}
By Proposition~\ref{Direct sum pm case}, 
it is enough to show $T_\p$-operator has no eigenform on $S_{k,l}(\GL_2(A))$ with eigenvalues $\pm P^{k/2}$. We give a complete proof only for $l\neq0$ and the proof is similar when $l=0$. So, we assume 
$l \neq 0$.

Since $\dim S_{k,l}(\GL_2(A))=2$,  $S_{k,l}(\GL_2(A))=\langle g_1^yh^l, g_1^{x}\Delta h^l\rangle$ for some $y\in \{q+1, \ldots, 2q+1\}$ such that $k=y(q-1)+l(q+1)$ and where $x := y-(q+1)$.  There are two cases to be considered.
We first assume that $(y,l)\ne (2q+1,1)$:
Recall the following $u$-expansions
\begin{equation*}
g_1^y = 
\begin{cases}
\dis \sum_{i=0}^{y} \bstwomat{y}{i}(-1)^{i}(T^q-T)^iu^{i(q-1)} + O(u^{(l+q)(q-1)}) \qquad \qquad  \mathrm{if} \ y<l+(q-1),\\
\dis \sum_{i=0}^{l+(q-1)} \bstwomat{y}{i}(-1)^{i}(T^q-T)^iu^{i(q-1)} + O(u^{(l+q)(q-1)}) \ \ \qquad \mathrm{if} \ y\geq l+(q-1).
\end{cases}
\end{equation*}

\begin{equation}
\label{u-series expansion of g_1^x in dimension 2 case}
g_1^x= \sum_{i=0}^x \bstwomat{x}{i}(-1)^i(T^q-T)^iu^{i(q-1)} + O(u^{(q-1)(q^2-q+1)}).
\end{equation}

\begin{equation}
\label{u-series expansion of h^l}
h^l= (-1)^lu^l + (-1)^l l u^{(q-1)^2+l} +(-1)^{l-1}l (T^q-T) u^{q(q-1)+l} + O(u^{(l+q)(q-1)+l}).
\end{equation}

\begin{equation*}
\Delta= -u^{q-1}+u^{q(q-1)} -(T^q-T)u^{(q+1)(q-1)} +O(u^{(q^2-q+1)(q-1)}).
\end{equation*}

\begin{equation}
\label{u-series expansion of Delta.h^l}
\small \Delta h^l = (-1)^{l+1}u^{q-1+l} +(-1)^l (1-l)u^{q(q-1)+l}+ (-1)^l (l-1) (T^q-T)u^{(q^2-1)+l} + O(u^{(l+q)(q-1)+l})
\end{equation}
Finally, we have the required $u$-expansions of  $g_1^{x}\Delta h^l$
and $g_1^yh^l$ as
\begin{equation*}
 g_1^{x}\Delta h^l=
\begin{cases}
\dis (-1)^{l+1}\sum_{i=1}^{x+1}\bstwomat{x}{i-1}(-1)^{i-1}(T^q-T)^{i-1}u^{i(q-1)+l} + O(u^{l(q-1)+l}) \quad \mathrm{if} \ x+1<l, \\
\dis (-1)^{l+1}\sum_{i=1}^{l-1}\bstwomat{x}{i-1}(-1)^{i-1}(T^q-T)^{i-1}u^{i(q-1)+l} + O(u^{l(q-1)+l}) \quad \mathrm{if} \ x+1\geq l.
\end{cases}
\end{equation*}
\begin{equation*}
g_1^yh^l = (-1)^l\sum_{i=0}^{l-1} \bstwomat{y}{i}(-1)^{i}(T^q-T)^iu^{i(q-1)+l} + O(u^{l(q-1)+l}).
\end{equation*}
We first  show that $T_\p(g_1^x\Delta h^l)\ne \pm P^{\frac{x(q-1)+(q^2-1)+l(q+1)}{2}}g_1^x\Delta h^l.$ The $(l+(q-1))$-th coefficient of $T_\p(g_1^x\Delta h^l)$ is given by (cf. \cite[Example 7.4]{Gek88})
\begin{equation}
\label{l+q-1 th coefficient of g1x Delta h Hecke action}
a_{T_\p(g_1^x\Delta h^l)}(l+(q-1))= \sum_{0\leq i < l+q-1} \bstwomat{l+q-2}{i} P^{l+q-1-i} a_{g_1^x\Delta h^l}(l+(i+1)(q-1)).
\end{equation}
For $f \in A, g\in A \setminus \{0\}$, $|\frac{f}{g}| := q^{\deg(f)-\deg(g)}$. Now, take
the norm of $a_{T_\p(g_1^x\Delta h^l)}(l+(q-1))$ to get 
\begin{align*}
 |a_{T_\p(g_1^x\Delta h^l)}(l+(q-1))| &\leq \max_{\substack{1\leq i \leq l+q-1}}\{|P^{l+q-i} a_{g_1^x\Delta h^l}(l+i(q-1))|\}\\
&= \max_{\substack{1\leq i \leq l+q-1}}\{|P^{l+q-i}\sum_{\substack{\alpha\in \N \cup \{0\}, \beta \in \N\\ \alpha + \beta =i}} a_{g_1^x}(\alpha(q-1))\cdot a_{\Delta h^l}(\beta(q-1)+l)|\}
\end{align*}
By~\eqref{u-series expansion of g_1^x in dimension 2 case} and~\eqref{u-series expansion of Delta.h^l}, we have 
$a_{g_1^x}(i(q-1))=0$ for $x<i\leq l+q-1$ and $a_{\Delta h^l}(\beta(q-1)+l)=0$ for $1\leq \beta \leq l+q-1$ with $\beta \notin \{1, q, q+1\}.$ Therefore, we get
\begin{align*}
|a_{T_\p(g_1^x\Delta h^l)}(l+(q-1))| &\leq \max_{\beta\in \{1, q, q+1\}} \Big\{ \max_{\substack{1\leq i \leq l+q-1, \\ 0\leq i-\beta \leq x}} \{|P^{l+q-i} a_{g_1^x}((i-\beta)(q-1)) a_{\Delta h^l}(\beta(q-1)+l)|\}\Big\}\\
&= \max \Big\{\max_{\substack{1\leq i \leq l+q-1\\ 0\leq i-1\leq x}}\{q^{i(q-1)+l}\}, \max_{\substack{1\leq i \leq l+q-1\\ 0\leq i-q\leq x}}\{q^{(i-q)(q-1)+l} \}, \max_{\substack{1\leq i \leq l+q-1\\ 0\leq i-q-1\leq x}}\{q^{(i-q)(q-1)+l}\} \Big\}\\
&= \max \Big\{q^{(x+1)(q-1)+l}, q^{x(q-1)+l}, q^{(x+1)(q-1)+l} \Big\}= q^{(x+1)(q-1)+l}.
\end{align*}
Hence, we have
\begin{equation}
 \label{bound for 2nd coefficient of T_p operator on (g1)^x Delta (h)^l}
|a_{T_\p(g_1^x\Delta h^l)}(l+(q-1))| \leq q^{(x+1)(q-1)+l}.
\end{equation}
On the other hand, the assumption $(y,l)\ne (2q+1,1)$ implies $(x+1)(q-1)+l< \frac{x(q-1)+(q^2-1)+l(q+1)}{2}.$
Since $a_{g_1^x\Delta h^l}(l+(q-1))= (-1)^{l+1}$, combining the last inequality with~\eqref{bound for 2nd coefficient of T_p operator on (g1)^x Delta (h)^l}, we get 
$T_\p(g_1^x\Delta h^l)\ne \pm P^{\frac{x(q-1)+(q^2-1)+l(q+1)}{2}}g_1^x\Delta h^l$.
Now, by the same technique, we give an upper bound on the coefficient $a_{T_\p(g_1^yh^l)}(l+(q-1)).$ Recall that,
\begin{equation}
a_{T_\p(g_1^yh^l)}(l+(q-1))= \sum_{0\leq i < l+(q-1)} \bstwomat{l+q-2}{i} P^{l+q-1-i} a_{g_1^yh^l}(l+(i+1)(q-1)).
\end{equation}
Now, take the norm of $a_{T_\p(g_1^yh^l)}(l+(q-1))$ to get 
\begin{align*}
 |a_{T_\p(g_1^yh^l)}(l+(q-1))| &\leq \max_{\substack{1\leq i \leq l+q-1}}\{|P^{l+q-i} a_{g_1^yh^l}(l+i(q-1))|\}\\
&= \max_{\substack{1\leq i \leq l+q-1}}\{|P^{l+q-i}\sum_{\substack{\alpha, \beta\in \N \cup \{0\}\\ \alpha + \beta =i}} a_{g_1^y}(\alpha(q-1))\cdot a_{h^l}(\beta(q-1)+l)|\}.
\end{align*}
By~\eqref{u-series expansion of h^l}, we get that $a_{h^l}(\beta(q-1)+l)=0$ for $0\leq \beta \leq l+q-1$ with $\beta\notin \{0, q-1, q\}.$
\begin{itemize}
 \item When $y<l+q-1$. In this case, $a_{g_1^y}(i(q-1))=0$ for $y< i\leq l+q-1.$ Hence,
\begin{align*}
|a_{T_\p(g_1^yh^l)}(l+(q-1))| &\leq \max_{\beta\in \{0, q-1, q\}} \Big\{ \max_{\substack{\beta\in \{0, q-1, q\}\\ 1\leq i\leq l+q-1 \\ 0\leq i-\beta \leq y}} \{|P^{l+q-i} a_{g_1^y}((i-\beta)(q-1)) a_{h^l}(\beta(q-1)+l)| \}\Big\}\\ 
 &= \max \Big\{\max_{\substack{1\leq i\leq l+q-1 \\ 0\leq i\leq y}}\{q^{i(q-1)+l+q}\}, \max_{\substack{1\leq i\leq l+q-1 \\ 0\leq i-q+1\leq y}}\{q^{(i-q)(q-1)+l+q}\},  \max_{\substack{1\leq i\leq l+q-1 \\ 0\leq i-q\leq y}}\{q^{(i-q)(q-1)+l+q}\}
 \Big\} \\
 &= \max \Big\{q^{y(q-1)+l+q}, q^{(y-1)(q-1)+l+q}, q^{y(q-1)+l+q} \Big\}= q^{y(q-1)+l+q}. 
\end{align*}
\item When $y\geq l+(q-1)$, a similar argument as above gives 
$|a_{T_\p(g_1^yh^l)}(l+q-1)| \leq q^{(l+q-1)(q-1)+l+q}.$
\end{itemize}
Finally, we get
\begin{equation}
\label{bound for 2nd coefficient of T_p operator on (g1)^y.h^l}
|a_{T_\p(g_1^yh^l)}(l+(q-1))| \leq \begin{cases}
q^{y(q-1)+l+q}         \quad \qquad \qquad \mathrm{if} \ y< l+(q-1),\\
 q^{(l+(q-1))(q-1)+l+q} \qquad  \mathrm{if} \ y\geq l+(q-1).
\end{cases}
\end{equation}
Since $T_\p(g_1^x\Delta h^l)\ne \pm P^{\frac{x(q-1)+(q^2-1)+l(q+1)}{2}}g_1^x\Delta h^l$, it is now enough to show that there does not exist any $c\in C$ such that 
\begin{equation}
\label{T_p on linear combination for contradiction}
T_\p(g_1^y h^l + c g_1^{x} \Delta h^l)= \pm P^{\frac{y(q-1)+l(q+1)}{2}}(g_1^y h^l + c g_1^{x}\Delta h^l)
\end{equation}
holds. On the contrary, suppose there is an element  $c\in C$ such that~\eqref{T_p on linear combination for contradiction} holds
with $``+"$ sign. A similar argument works with $``-"$ sign as well.
The $l$-th coefficients of $T_\p(g_1^yh^l)$ and $T_\p(g_1^x\Delta h^l)$ are given by (cf. \cite[Example 7.4]{Gek88})
\begin{equation*}
a_{T_\p(g_1^yh^l)}(l)
= (-1)^l \sum_{0\leq j < l} \bstwomat{l-1}{j}P^{l-j} \bstwomat{y}{j}(-1)^{j}(T^q-T)^j,
\end{equation*}
\begin{equation*}
a_{T_\p(g_1^x\Delta h^l)}(l)=
\begin{cases}
\dis (-1)^{l+1}\sum_{j=1}^{x+1}\bstwomat{l-1}{j}P^{l-j}\bstwomat{x}{j-1}(-1)^{j-1}(T^q-T)^{j-1} 
                               \quad \mrm{if} \ x+1< l,\\
\dis (-1)^{l+1}\sum_{j=1}^{l-1}\bstwomat{l-1}{j}P^{l-j}\bstwomat{x}{j-1}(-1)^{j-1}(T^q-T)^{j-1} 
                                \quad \mrm{if} \ x+1\geq l.
\end{cases}
\end{equation*}
Comparing the $l$-th coefficients on both sides of~\eqref{T_p on linear combination for contradiction}, we get
\begin{equation}
\label{comparing the l-th coefficients of T_p}
\small \sum_{0\leq j < l} \bstwomat{l-1}{j}P^{l-j} \bstwomat{y}{j}(-1)^{j}(T^q-T)^j - c\sum_{j=1}^{x_0}\bstwomat{l-1}{j}P^{l-j}\bstwomat{x}{j-1}(-1)^{j-1}(T^q-T)^{j-1}= P^{\frac{y(q-1)+l(q+1)}{2}},
\end{equation}
where $x_0:= \mrm{min}\{ x,l-2\}+1$.
If $c\sum_{j=1}^{x_0}\bstwomat{l-1}{j}P^{l-j}\bstwomat{x}{j-1}(-1)^{j-1}(T^q-T)^{j-1}=0$, then the inequality $lq<\frac{y(q-1)+l(q+1)}{2}$ would imply that both sides of~\eqref{comparing the l-th coefficients of T_p} have different degrees. So, the term  $c\sum_{j=1}^{x_0}\bstwomat{l-1}{j}P^{l-j}\bstwomat{x}{j-1}(-1)^{j-1}(T^q-T)^{j-1} \neq 0.$ 
Let
$j_{\max}:= \max \{1\leq j \leq x_0| \bstwomat{l-1}{j}\bstwomat{x}{j-1}\neq 0\}.$ Then,
$|\sum_{j=1}^{x_0}\bstwomat{l-1}{j}P^{l-j}\bstwomat{x}{j-1}(-1)^{j-1}(T^q-T)^{j-1}|=q^{j_{\max}(q-1)+l-q}.$ Since $lq<\frac{y(q-1)+l(q+1)}{2}$, it follows that 
$$|P^{\frac{y(q-1)+l(q+1)}{2}}- \sum_{0\leq j < l} \bstwomat{l-1}{j}P^{l-j} \bstwomat{y}{j}(-1)^{j}(T^q-T)^j|=q^\frac{y(q-1)+l(q+1)}{2}.$$
Therefore,~\eqref{comparing the l-th coefficients of T_p} gives us
\begin{equation}
\label{value of the constant}
c= -\frac{P^{\frac{y(q-1)+l(q+1)}{2}}- \sum_{0\leq j < l} \bstwomat{l-1}{j}P^{l-j} \bstwomat{y}{j}(-1)^{j}(T^q-T)^j}{\sum_{j=1}^{j_{\max}}\bstwomat{l-1}{j}P^{l-j}\bstwomat{x}{j-1}(-1)^{j-1}(T^q-T)^{j-1}}\in K,
\end{equation}  hence
$ |c|= q^{\frac{y(q-1)+l(q+1)}{2}-(j_{\max}(q-1)+l-q)}.$
Note that $a_{(g_1^y h^l + c g_1^{x} \Delta h^l)}(l+(q-1))= (-1)^{l+1}y(T^q-T)+(-1)^{l+1} c.$
Using the inequality $lq<\frac{y(q-1)+l(q+1)}{2}$, from~\eqref{value of the constant} we obtain
\begin{equation}
|a_{(g_1^y h^l + c g_1^{x} \Delta h^l)}((q-1)+l)|= q^{\frac{y(q-1)+l(q+1)}{2}-(j_{\max}(q-1)+l-q)}.
\end{equation}
Comparing $(q-1)+l$-th coefficients on both sides of~\eqref{T_p on linear combination for contradiction} we get
\begin{equation}
\label{absolute valuation for (q-1)+l th coefficient of T_p of linear combination}
|a_{T_\p(g_1^y h^l + c g_1^{x} \Delta h^l)}((q-1)+l)| = q^{y(q-1)+l(q+1)-(j_{\max}(q-1)+l-q)}.
\end{equation}
On the other hand, from \eqref{bound for 2nd coefficient of T_p operator on (g1)^y.h^l} we have
\begin{align*}
|a_{T_\p(g_1^y h^l + c g_1^{x} \Delta h^l)}((q-1)+l)| 
&\leq \max \{|a_{T_\p(g_1^y h^l)}((q-1)+l)|, |c|\cdot| a_{T_\p(g_1^{x} \Delta h^l)}((q-1)+l)|\}\\
&\leq \max\{q^{y_0(q-1)+l+q}, q^{\frac{y(q-1)+l(q+1)}{2}-(j_{\max}(q-1)+l-q)}\cdot q^{(x+1)(q-1)+l} \}
\end{align*}
where $y_0:= \mrm{min}\{y, l+(q-1) \}.$
Since $0\leq j_{\max} <l$, an easy verification shows that
$q^{y_0(q-1)+l+q}<q^{y(q-1)+l(q+1)-(j_{\max}(q-1)+l-q)}. $
Moreover, the inequality $(x+1)(q-1)+l< \frac{x(q-1)+(q^2-1)+l(q+1)}{2}$ implies  
$q^{\frac{y(q-1)+l(q+1)}{2}-(j_{\max}(q-1)+l-q)}\cdot q^{(x+1)(q-1)+l}<q^{y(q-1)+l(q+1)-(j_{\max}(q-1)+l-q)}.$
Therefore, we can conclude 
$$|a_{T_\p(g_1^y h^l + c g_1^{x} \Delta h^l)}((q-1)+l)|<q^{y(q-1)+l(q+1)-(j_{\max}(q-1)+l-q)},$$
which contradicts~\eqref{absolute valuation for (q-1)+l th coefficient of T_p of linear combination}. 
Hence, the $T_\p$-operator has no eigenform on $S_{k,l}(\GL_2(A))$ with eigenvalue $\pm P^{k/2}$, and the result now follows from Proposition~\ref{Direct sum pm case}.

We now consider the case when $(y,l) = (2q+1,1)$. In this case,  $S_{k,l}(\GL_2(A))=\langle g_1^{2q+1}h, g_1^{q}\Delta h\rangle$. By the $u$-series expansion, we get $a_{T_{\p}(g_1^{2q+1}h)}(1)=-P, a_{T_\p(g_1^q\Delta h)}(1)=0$ and $a_{T_\p(g_1^q\Delta h)}(q)=P^q$. This implies that, for any $(c_1,c_2) \in C^2 \setminus \{ (0,0)\}$, $T_\p(c_1 g_1^{2q+1}h+c_2 g_1^q\Delta h)\ne P^{q^2}(c_1 g_1^{2q+1}h+c_2 g_1^q\Delta h)$. This can be checked by comparing the $1$-st coefficient if $c_1 \ne 0$,
the $q$-th coefficient if $c_1=0$. Now, we are done by Proposition~\ref{Direct sum pm case}.

Finally, consider the case $l=0$: Here, the proof is exactly similar to $l\neq 0$, except that  
we need to consider $(q-1)$-th, $2(q-1)$-th coefficients and use 
the inequality $(x+2)(q-1)<\frac{x(q-1)+2(q^2-1)}{2}$.
\end{proof}

\subsection{Evidences to Question~\ref{example of direct sum at level P} for prime ideals $\p$:}
In this section, we give evidences in the support of  Question~\ref{example of direct sum at level P} 
for prime ideals $\p$. In this direction, we  need a proposition, which is a generalization of a result of Gekeler (cf.~\cite[Corollary 7.6]{Gek88}),
where he proved that $T_\p h = Ph$ for any prime ideal $\p=(P)$.
We now show that this result continues to hold for $f\in M_{k,1}(\Gamma_0(\mfm))$ with $a_f(1)\ne 0$.  
\begin{prop}\label{Type 1 eigenforms}
Suppose the $u$-series expansion of  $f\in M_{k,1}(\Gamma_0(\mfm))$ at $\infty$
is given by $\sum_{j=0}^\infty a_f(j(q-1)+1)u^{j(q-1)+1}$ with $a_f(1)\ne 0$. 
If $T_\p f = \lambda f$ for some $\lambda\in C$, then $\lambda=P.$
In particular, 
$T_\p f =   P^{\frac{k}{2}} f$ can happen only when  $k=2$.
\end{prop}
\begin{proof}
Let $G_{i,P}(X)$ denote the $i$-th Goss polynomial corresponding 
to the lattice $\Lambda_P= \ker(\rho_P)=\{x\in C \mid \rho_P(x)=0\}$,
where $\rho_P$ is the Carlitz module with value at $P$. 
By~\cite[Proposition 5.2]{Arm11} (the normalization here is different from there), we have 
\begin{equation}
\label{coefficient of u}
T_\p f  = P^k\sum_{j\geq 0}a_f(j(q-1)+1)(u_\p)^{j(q-1)+1} + \sum_{j\geq 0}a_f(j(q-1)+1)G_{j(q-1)+1,P}(P u),
\end{equation}
where $u_\p(z)=u(Pz)=u^{q^d}+\cdots.$
To determine $\lambda$,  we wish to compute the coefficient of $u$ 
in the $u$-series expansion of $T_\p f$. In~\eqref{coefficient of u}, the term involving $u_\p$ does not contribute
to the coefficient of $u$.
By~\cite[Proposition 3.4(ii)]{Gek88}, we know that
$$ G_{i,P}(X)=X(G_{i-1,P}(X)+\alpha_1 G_{i-q,P}(X)+\cdots + \alpha_j G_{i-q^j,P}(X)+\cdots).$$
In $G_{i,P}(Pu)$, the coefficient of $u$ in $G_{j(q-1)+1,P}(P u)$ is $0$ for $j>0$.
Since $G_{1,P}(X)=X$ (cf. \cite[Proposition 3.4(v)]{Gek88}), we can conclude that 
$T_\p f= Pa_f(1)u+ \mathrm{higher\ terms}.$
By comparing the coefficient of $u$ on both sides, we get $\lambda=P$.
\end{proof}
\begin{remark}
In Proposition~\ref{Type 1 eigenforms}, Goss polynomials, which occur as the coefficients of $T_\p f$,
are very difficult to handle if $l \neq 1$ (cf. \eqref{coefficient of u} and \cite[Proposition 5.2]{Arm11}). 
So, we have restricted ourselves to $l=1$ in the last proposition.

\end{remark}
We now give some instances where Question~\ref{example of direct sum at level P} for prime ideals $\p$
has an affirmative answer.
\begin{prop}
\label{exple1}
For any prime ideal $\p$, Question~\ref{example of direct sum at level P} is true for level $\p$ in the following cases:
\begin{enumerate}[(1)]
\item   \begin{enumerate}
     \item $1\leq l \leq q-2$ and $k=2l+\alpha(q-1)$ where $\alpha \in \{0,\ldots, l\}$,
     \item $l=0$ and $k=\beta(q-1)$ where $\beta\in \{1, \ldots , q+1\}$,
      \item $l=1$ and $k=\alpha(q-1)+(q+1)$ where $\alpha\in \{0,\ldots, q\}$.
    \end{enumerate}
\item $k\leq 3q.$
\end{enumerate}
\end{prop}
\begin{proof}
Note that in all of these cases $\dim S_{k,l}(\GL_2(A))\leq 1.$ Hence the $T_\p$-operator is diagonalizable on $S_{k,l}(\GL_2(A)).$  As in our earlier discussion, 
Question~\ref{example of direct sum at level P} has an affirmative answer for $\p$ 
if we show that $\ker(T_\p)=0$ and the $T_\p$-operator has no eigenform on $S_{k,l}(\GL_2(A))$ with eigenvalues $\pm P^{\frac{k}{2}}$.
We prove these statements in all cases.

\begin{enumerate}[(1)]
\item
 \begin{enumerate}
      \item  Since $l>0$, $M_{k,l}(\GL_2(A))=S_{k,l}(\GL_2(A))$. 
        \begin{itemize}
         \item If $\alpha \in \{0,\ldots, l-1\}$, then $\dim S_{2l+\alpha(q-1),l}(\GL_2(A))=0$ and the result follows trivially.  
         \item If $\alpha=l$, then $\dim S_{2l+\alpha(q-1),l}(\GL_2(A))=1$ and $S_{2l+\alpha(q-1),l}(\GL_2(A))= \langle h^l\rangle$.             
         By~\cite[(9)]{JP14} (or by \cite[Theorem 3.17]{Pet13}), the $T_\p$-operator acts on $h^i$ by $P^i$
             for $1\leq i\leq q-2.$ Since $P^l\ne \pm P^{\frac{l(q+1)}{2}}$ for $1\leq l \leq q-2$ the result follows.
       \end{itemize}

\item When $l=0$, we prove the required claim in two steps.
      \begin{itemize}
       \item  For $\beta\in \{1, \ldots , q\}$, $M_{\beta(q-1),0}(\GL_2(A))=\langle g_1^\beta 
              \rangle$. Therefore,
              $S_{\beta(q-1),0}(\GL_2(A))=\{0\}$ and the result follows. 
 
       \item  If $\beta = q+1$, $S_{q^2-1,0}(\GL_2(A))=\langle \Delta\rangle.$  
              By~\cite[Corollary 7.5]{Gek88}, we have $T_\p(\Delta)= P^{q-1}\Delta$. Since $P^{q-1}\ne \pm P^{\frac{q^2-1}{2}}$, the result follows.
       \end{itemize}
       
\item If $\alpha \in \{0,\ldots, q\} $, $S_{k,1}(\GL_2(A)) = \langle g_1^\alpha h \rangle$. Since $a_{g_1^\alpha h}(1)\ne 0$, by Proposition~\ref{Type 1 eigenforms}, we have $T_\p(g_1^\alpha h)= P g_1^\alpha h$, the result follows.

\end{enumerate}

\item Let $0\leq l \leq q-2.$ If $k\not \equiv 2l \pmod {q-1}$, then $M_{k,l}(\GL_2(A))=\{0\}$ and Question~\ref{example of direct  
      sum at level P} is trivially true. So, we only consider the cases $k \equiv 2l \pmod {q-1}$ i.e $k=2l+x(q-1)$ for some $x\in \N\cup \{0\}$.
      The condition $k \leq 3q$ implies $x\leq 4$. 
      \begin{itemize}
         \item If $x<l$, then $\dim M_{k,l}(\GL_2(A))=0$ and the result follows. 
         \item If $x=l$, then $k=l(q+1)$. If $l \neq 0$, then $S_{l(q+1), l}(\GL_2(A))=\langle h^l \rangle$.
               So, we are back to case 1(a). If $l=0$, then $S_{0,0}=\{0\}$ and the result follows.
      \end{itemize}
      Therefore, the remaining cases of interest are $l < x \leq 4$. If $l\geq 2$, the inequality 
      $k\leq 3q$ forces that $x\leq 2$ and we are back to the case $x\leq l$.
       So, it is enough to consider for $l\in \{0,1\}$ with $ l < x \leq 4$.
     \begin{itemize}
          \item For $l=0$: If $(q,x)\ne (3,4)$, then $M_{x(q-1),0}(\GL_2(A)) = \langle g_1^x \rangle$ 
                and $S_{x(q-1),0}(\GL_2(A))=\{0\}$ and the result follows. 
                If $(q,x)=(3,4)$, then $k=(q+1)(q-1)$, we are 
                back to  case $1(b)$.
          \item For $l=1$: we have $k=(x-1)(q-1)+(q+1)$ where $1<x\leq 3.$ Since $q\geq 3$, we are 
                back to  case $1(c)$.
     \end{itemize}
\end{enumerate}
This completes the proof of the proposition.
\end{proof}
We remark that our Proposition~\ref{exple1} is similar to Theorem 5.8, Corollary 5.11 and Theorem 5.14 of~\cite{BV19a} for $\p= (T)$-case.
In a contrast to  Proposition~\ref{exple1}, 
in the next proposition, we consider the situation with $\m\ne A$ and $\dim S_{k,l}(\Gamma_0(\m))=2$ 
satisfying Question~\ref{example of direct sum at level P} for level $\p\m$.
%
 
\begin{prop}
\label{exple2}
For $\deg m=1$ and $\p\nmid \m$, Question~\ref{example of direct sum at level P} is true for level  $\p\m$ 
when
\begin{enumerate}[(i)]
\item  $l> \frac{q-1}{2}$ and $k=2l-(q-1)$, or
\item  $l=1$ and $k=q+1$.
\end{enumerate}
 \end{prop}
%
\begin{proof}
 We may assume that $\m=(T)$, since a similar calculation works for any ideal $\m$ with $\deg m=1$. We now follow the strategy as in the proof of the Proposition~\ref{exple1}.
 \begin{enumerate}[(i)]
\item In this case, $S_{k,l}(\Gamma_0(T))=\{0\}$ (cf.~\cite[Proposition 4.1]{DK2}) and the result
         follows trivially.
\item      First, we show that the operator $T_\p-P$ is zero  on $S_{q+1,1}(\Gamma_0(T))$.
      Recall that, $\Delta_T(z) := \frac{g_1(Tz)-g_1(z)}{T^q-T}, \Delta_W(z) := \frac{T^qg_1(Tz)-Tg_1(z)}{T^q-T}\in M_{q-1,0}(\Gamma_0(T))$.
      
      By~\cite[Proposition 4.3(3)]{DK2}, $\dim_C\ S_{q+1,1}(\Gamma_0(T)) = 2$ and a basis is given by  
      $\{\Delta_T E_T$, $\Delta_W E_T\}$. By~\cite[Proposition 4.3(8)]{DK2}), $h= -\Delta_W E_T$.
      Since $T_\p h= Ph$, we obtain $T_\p(\Delta_W E_T)= P\Delta_W E_T$. Note that $\Delta_T = -T^{-1}\Delta_W|W_T^{(T)}$ and $T_\p W_T^{(T)}=W_T^{(T)} T_\p$ (cf. Theorem~\ref{T_P and A-L commutes}), using $E_T|_{2,1}W_T^{(T)}=-E_T$ (cf. \cite[Proposition 3.3]{DK1}),  we get
      \begin{align*}
        T_\p(\Delta_T E_T) & = T_\p((T^{-1}\Delta_W E_T)|W_T^{(T)}) = (T_\p(T^{-1}\Delta_W E_T))|W_T^{(T)} \\
                           & = T^{-1} ( P\Delta_W E_T)|W_T^{(T)}  = P\Delta_TE_T.
      \end{align*}
     Thus, $T_\p \equiv P$ on $S_{q+1,1}(\Gamma_0(T))$. So, $T_\p$-operator is injective, diagonalizable on $S_{q+1,1}(\Gamma_0(T))$, which proves 
    Question~\ref{example of direct sum at level P}(1). Question~\ref{example of direct sum at level P}(2)  follows from Proposition~\ref{Direct sum pm case}.     
      Finally, Question~\ref{example of direct sum at level P}(3) follows from the diagonalizability of $T_\p$-operator on $S_{q+1, 1}(\Gamma_0(T))$.
    \end{enumerate}
\end{proof}
\subsection{Counterexample to Question~\ref{example of direct sum at level P}(2)}
In the section,  we show that the direct sum decomposition~\eqref{direct sum decomposition level pm in corollary} does not hold if $\m\ne A$ and $(k,l)=(2,1).$  We first prove a result which is of independent interest.
\begin{lem}
	\label{Eigenvalue of E_P}
	Let $\p_1, \p_2$ be two distinct prime ideals of $A$ generated by monic irreducible polynomials $P_1, P_2$, respectively. Then,  $T_{\p_1} E_{P_2}= P_1E_{P_2}.$
\end{lem}

\begin{proof}
	By~\cite[(8.2)]{Gek88}, the function $E(z)=\sum_{a\in A_+}au(az)$, where $A_+$ denotes the set of all monic polynomials in $A.$ Hence,
$	E_{P_2}(z) = \sum_{a\in A_+}au(az)- P_2\sum_{a\in A_+}au(P_2az) = \sum_{a\in A_+,\ P_2\nmid a} au(az).$
	We now use an argument in the proof of~\cite[Theorem 2.3]{Pet13} to get
	\begin{align*}
	T_{\p_1} E_{P_2} 
	&= \sum_{\substack{Q\in A\\ \deg Q<\deg P_1}}E_{P_2} \left(\frac{z+Q}{P_1}\right) +P_1^2 E_{P_2}(P_1z)\\
	&= \sum_{\substack{Q\in A\\ \deg Q<\deg P_1}} \sum_{\substack{a\in A_+\\ P_2\nmid a}} au\left(a\frac{z+Q}{P_1}\right) + P_1^2 \sum_{\substack{a\in A_+\\ P_2\nmid a}} au(P_1az)\\
	&= \frac{1}{\tilde{\pi}} \sum_{\substack{Q\in A\\ \deg Q<\deg P_1}} \sum_{\substack{a\in A_+\\ P_2\nmid a}} \sum_{b\in A} \frac{aP_1}{az+aQ+P_1b} + P_1 \sum_{\substack{a\in A_+\\ P_2\nmid a}} P_1a u(P_1az)\\ 
	&= \frac{1}{\tilde{\pi}}  \sum_{\substack{a\in A_+\\ P_2\nmid a}} aP_1 \sum_{b\in A} \sum_{\substack{Q\in A\\ \deg Q<\deg P_1}} \frac{1}{az+aQ+P_1b} + P_1 \sum_{\substack{a\in A_+\\ P_2\nmid a}} P_1a u(P_1az)\\
	&= P_1\sum_{\substack{a\in A_+\\ P_1P_2\nmid a}}au(az) + P_1\sum_{\substack{a\in A_+\\ P_2\nmid a}}P_1au(P_1 az) = P_1 \sum_{\substack{a\in A_+\\ P_2\nmid a}} au(az)= P_1 E_{P_2}.
	\end{align*}
	This completes the proof of the Lemma.
\end{proof}
We now show that, if $\m\ne A$ and $(k,l)=(2,1)$, then there are non-zero Drinfeld cusp forms which are both $\p$-oldforms and $\p$-newforms.

\begin{prop}
\label{counter example for direct sum decomposition for square free p-part}
Suppose $\m\ne A$. For any prime ideal $\p \nmid \m$,  we have 
$$S_{2,1}^{\p-\old}(\Gamma_0(\p\m))\cap S_{2,1}^{\p-\new}(\Gamma_0(\p\m))\ne\{0\}.$$

\end{prop}
\begin{proof}
Let $\p_2$ be a prime divisor of $\m$ generated by a monic irreducible polynomial $P_2.$ Clearly, $0 \neq E_{P_2}-\delta_P E_{P_2} \in S_{2,1}(\Gamma_0(\p\m)).$  We show that
$E_{P_2}-\delta_P E_{P_2} \in S_{2,1}^{\p-\old}(\Gamma_0(\p \m))\cap S_{2,1}^{\p-\new}(\Gamma_0(\p \m)).$
By definition, $E_{P_2}-\delta_P E_{P_2} \in S_{2,1}^{\p-\old}(\Gamma_0(\p \m))$.
Combining~\eqref{Trace and delta P},~\eqref{delta_P and A-L together} and Lemma~\ref{Eigenvalue of E_P}, we get
\begin{equation}
\label{To show E_T-delta_p E_T is an old form trace 0}
\Tr_\m^{\p \m}(E_{P_2}-\delta_P E_{P_2}) =E_{P_2}-P^{-1}T_{\p}(E_{P_2}) = E_{P_2}-E_{P_2} =0.
\end{equation}
By~\eqref{delta_P and A-L together}, we deduce that 
\begin{equation*}
\label{To show E_T-delta_p E_T is an old form trace A-L 0}
\Tr_{\m}^{\p \m}((E_{P_2}-\delta_P E_{P_2})|W_\p^{(\p \m)}) = \Tr_{\m}^{\p \m}(E_{P_2}|W_{\p}^{(\p \m)}-(\delta_{P}E_{P_2})|W_{\p}^{(\p \m)})
= \Tr_{\m}^{\p \m}(\delta_{P}E_{P_2} -E_{P_2})=0.
\end{equation*}
This proves that $E_{P_2}-\delta_P E_{P_2} \in S_{2,1}^{\p-\new}(\Gamma_0(\p \m))$. The result follows.
\end{proof}

\begin{remark}
\label{counter example higher powers remark}
For $f\in S_{k,l}(\Gamma_0(\n))$, $T_\p(f^{q^n})=(T_\p(f))^{q^n}$
for any $n \in \N$. An argument similar to Proposition~\ref{counter example for direct sum decomposition for square free p-part} gives us 
\begin{equation}
\label{counter example higher degree}
0\ne E_{P_2}^{q^n}-{P^{q^n-1}} \delta_P E_{P_2}^{q^n} \in S_{2{q^n},1}^{\p-\old}(\Gamma_0(\p \m))\cap S_{2{q^n},1}^{\p-\new}(\Gamma_0(\p \m)). 
\end{equation}
Since $E$ behaves like a classical weight $2$ Eisenstein series, 
we believe that the phenomenon in~\eqref{counter example higher degree} may not happen for $l\ne 1$. 
\end{remark}
Proposition~\ref{counter example for direct sum decomposition for square free p-part} and Remark~\ref{counter example higher powers remark} 
imply that either one needs to reformulate the definition of $\p$-newforms for level $\p\m$
or exclude the cases above in formulating Question~\ref{example of direct sum at level P} for level $\p\m$.

\section{Oldforms and Newforms for square-free level $\n$}
\label{Oldforms and Newforms for square free level n}
In this section, we propose a definition of oldforms and newforms for Drinfeld modular forms of square-free level. We show that these spaces are invariant under the action of the Hecke operators. Throughout this section, we  assume that $\n$ is a square-free ideal of $A$ generated by a (square-free) monic polynomial $n\in A$.  Let $\p, \p_1$ be two prime ideals of $A$ generated by monic irreducible polynomials $P,P_1\in A$, respectively.

\begin{dfn}[Oldforms]
The space of oldforms of weight $k$, type $l$, and square-free level $\n$ is defined as
\begin{equation*}
S_{k,l}^{\old}(\Gamma_0(\n)):= \sum\limits_{\p|\n}(\delta_1,\delta_P)((S_{k,l}(\Gamma_0(\n/\p)))^2).
\end{equation*}
\end{dfn}
%
The lack of Petersson inner product for Drinfeld modular forms makes it difficult to define newforms. 
For classical modular forms, it is well-known that newforms can be characterized in terms of kernels of the Trace and twisted Trace operators (cf. \cite{Ser73}, \cite{Li75} for more details).
In this section, for Drinfeld modular forms, we adopt this approach to define newforms and  investigate their properties.

\begin{dfn}[Newforms]
\label{definition of newforms}
The space of newforms of weight $k$, type $l$, and square-free level $\n$ is defined as
$$S_{k,l}^{\new}(\Gamma_0(\n)):= \bigcap\limits_{\p|\n} (\Ker(\Tr_\frac{\n}{\p}^{\n})\cap \Ker({\Tr^\prime}_\frac{\n}{\p}^{\n})),
\ \mathrm{where} \ 
{\Tr^\prime}_\frac{\n}{\p}^{\n}f= {\Tr}_\frac{\n}{\p}^{\n}(f|_{k,l}W^{(\n)}_{\p}).$$
\end{dfn}

Next, we study the action of Hecke operators on $S_{k,l}^{\old}(\Gamma_0(\n))$, $S_{k,l}^{\new}(\Gamma_0(\n))$.
This depends on the commutativity of the (partial) Atkin-Lehner operators with the $T_\p$ and $U_\p$-operators.
In~\cite[Theorem 1.1]{Val22}, the author studied the commutativity of the (partial) Atkin-Lehner operators and the $T_\p$-operator and proved the following result.
\begin{thm}
\label{T_P and A-L commutes}
Let $\n,\p \subseteq A$ be ideals such that $\p \nmid \n$ and $\p$ is prime. 
For any ideal $\mathfrak{d}$ of $A$ such that $\mathfrak{d}||\n$, 
the actions of $T_{\p}W_{\mathfrak{d}}^{(\n)}$ and  $W_{\mathfrak{d}}^{(\n)}T_{\p}$ 
on $S_{k,l}(\Gamma_0(\n))$
are equal.
\end{thm}
We now study the commutativity of certain (partial) Atkin-Lehner operators and the $U_\p$-operator.  
The following result can be thought of as a generalization of Theorem~\ref{T_P and A-L commutes} to the $U_\p$-operator. Note that, Theorem~\ref{U_P and A-L commutes} holds for any integral ideal $\n$.
\begin{thm}
\label{U_P and A-L commutes}
Assume that $\p^\alpha|| \n$ for some $\alpha\in \N$.
For all prime divisors $\p_1$ of $\n$ with $\p_1 \neq \mfp$,  the actions of $U_{\p_1}W_{\p^\alpha}^{(\n)}$  and $W_{\p^\alpha}^{(\n)}U_{\p_1}$ 
on $S_{k,l}(\Gamma_0(\n)) $ are equal.
\end{thm}
\begin{proof}
By definition we have
\begin{equation*}
P_1^{l-k}U_{\p_1}W_{\p^\alpha}^{(\n)} = \sum_{\substack{Q\in A\\ \deg Q < \deg P_1}}\psmat{P^\alpha}{b}{n}{P^\alpha d}\psmat{1}{Q}{0}{P_1}= \sum_{\substack{Q\in A\\ \deg Q < \deg P_1}} \psmat{P^\alpha}{P^\alpha Q+bP_1}{n}{nQ+P^\alpha P_1d},
\end{equation*}
\begin{equation*}
P_1^{l-k}W_{\p^\alpha}^{(\n)}U_{\p_1} = \sum_{\substack{Q\in A\\ \deg Q < \deg P_1}}\psmat{1}{Q}{0}{P_1}\psmat{P^\alpha}{b}{n}{P^\alpha d}= \sum_{\substack{Q\in A\\ \deg Q < \deg P_1}} \psmat{P^\alpha+Q n}{b+P^\alpha Q d}{nP_1}{P^\alpha P_1d}.
\end{equation*}
To prove the proposition, it suffices to show that for any $Q\in A$ with $\deg Q<\deg P_1$, there exists a unique $Q^\prime\in A$ with $\deg Q^\prime<\deg P_1$ such that 
\begin{equation}
\label{matrix relation to show U_p and A-L commutes}
\psmat{P^\alpha+Q n}{b+P^\alpha Q d}{nP_1}{P^\alpha P_1d}= \psmat{x}{y}{z}{w}\psmat{P^\alpha}{P^\alpha Q^\prime+bP_1}{n}{nQ^\prime+P^\alpha P_1d},
\end{equation}
for some $\psmat{x}{y}{z}{w}\in \Gamma_0(\n)$. For any $Q,Q^\prime\in A$, \eqref{matrix relation to show U_p and A-L commutes} implies $x,w\in A$, $z\in \n$ and 
\begin{equation}
\label{equation for y}
-P_1 y= P^{\alpha}Q^\prime -P^{\alpha}Qd -b  + (n Q Q^\prime+bP_1 + \frac{n}{P^\alpha} bP_1 Q).
\end{equation}
Thus we are reduced to show that for any $Q\in A$ with $\deg Q<\deg P_1$, there exists a unique $Q^\prime\in A$ with $\deg Q^\prime<\deg P_1$  such that $y\in A$. 

Since $P_1|n$, we have $P_1|(n Q Q^\prime+bP_1 + \frac{n}{P^\alpha} bP_1 Q)$. Now it is enough to show that there exists a unique $Q^\prime\in A$ with $\deg Q^\prime<\deg P_1$ such that $P_1\mid P^\alpha(Q^\prime - Qd)-b$. 

Recall that, 
$P^\alpha d-b\frac{n}{P^\alpha}=1$. Since $P_1$ divides $\frac{n}{P^{\alpha}}$,
we get $QP^\alpha d \equiv Q \pmod {P_1}$ for any $Q \in A.$
So, it is enough to show that there exists a unique $Q^\prime\in A$
such that $P_1\mid P^\alpha  Q^\prime - (Q+b)$. Since $(P^\alpha,P_1)=1$,  the congruence $P^\alpha  f(X) \equiv  (Q+b) \pmod {P_1}$  has a unique solution in $A$ with $\deg(f(X))<\deg P_1$.
We are done.
\end{proof}
We are now ready to state the main theorem of this section.   

\begin{thm}
\label{Oldspace and Newspace are stable under Hecke operators}
The spaces $S_{k,l}^{\old}(\Gamma_0(\n))$, $S_{k,l}^{\new}(\Gamma_0(\n))$ are invariant under  
the action of the Hecke operators $T_\mfp$ for  $\mfp \nmid \n$ and $U_\p$ for $\p\mid \n$.
\end{thm}

\begin{proof}
Let $\p$ be a prime ideal of $A$ such that $\p\mid \n$. We first  show that the space $S_{k,l}^{\new}(\Gamma_0(\n))$ is stable under the $U_\p$-operator.
Let $\p_1 \ne \p$ be a prime divisor of $\n$ and $f\in S_{k,l}^{\p_1-\new}(\Gamma_0(\n))$.
Theorem~\ref{U_P and A-L commutes} (resp., Proposition~\ref{U_p, U_Q commutes and U_P, T_Q commutes})
implies that the $U_\p$-operator commutes with the $W_{\p_1}^{(\n)}$-operator  (resp., the $U_{\p_1}$-operator). 
Since $f\in S_{k,l}^{\p_1-\new}(\Gamma_0(\n))$, from Proposition~\ref{relation between Trace, A-L and U_P operator for level pm},
we obtain 
$$\Tr_\frac{\n}{\p_1}^{\n}(U_\p(f)) = U_\p(f)+ P_1^{-l}U_{\p_1}(U_\p(f)|W_\p^{(\n)}) =U_\p(\Tr_\frac{\n}{\p_1}^{\n}(f))=0.$$ 
A similar argument shows that ${\Tr^\prime}_\frac{\n}{\p_1}^{\n}(U_\p(f))=0$.
Thus $S_{k,l}^{\p_1-\new}(\Gamma_0(\n))$ is stable under the $U_\p$-operator.
Since the space $S_{k,l}^{\p-\new}(\Gamma_0(\n))$ is stable under the action of the $U_{\p}$-operator (cf.~\cite[Proposition 2.15]{BV20}), 
the space $S_{k,l}^{\new}(\Gamma_0(\n))$ is stable under the action of the $U_\p$-operator.

Next,  we show that the space $S_{k,l}^{\old}(\Gamma_0(\n))$ is stable under the action of the $U_\p$-operator.
Let $\p_1 \ne \p$ be a prime divisor of $\n$.
Let $\psi, \varphi\in S_{k,l}(\Gamma_0(\frac{\n}{\p_1}))$.
Since $\p|\frac{\n}{\p_1}$, we have  $U_\p(\psi), U_\p(\varphi)\in S_{k,l}(\Gamma_0(\frac{\n}{\p_1}))$. Moreover,~\eqref{delta_P and A-L together} and Theorem~\ref{U_P and A-L commutes} yield 
$$U_\p(\delta_{P_1} \varphi) = U_\p(\varphi|W_{P_1}^{(\n)})  = (U_\p(\varphi))|W_{P_1}^{(\n)}
                             = \delta_{P_1}(U_\p(\varphi)).$$
Hence for all $\p_1\mid \n$ with $\p_1\ne \p$ we have $U_\p(\psi +\delta_{P_1}\varphi)= 
U_\p(\psi) + \delta_{P_1} U_\p(\varphi)$ with $ U_\p(\psi), U_\p(\varphi) \in S_{k,l}(\Gamma_0(\frac{\n}{\p_1})).$
Since the space $S_{k,l}^{\p-\old}(\Gamma_0(\n))$ is stable under the action of the $U_\p$-operator (cf.~\cite[Proposition 2.15]{BV20}), the space $S_{k,l}^{\old}(\Gamma_0(\n))$ is stable under the action of the $U_\p$-operator.

An argument similar to the above would also imply that 
the spaces $S_{k,l}^{\new}(\Gamma_0(\n))$ and $S_{k,l}^{\old}(\Gamma_0(\n))$ are stable 
under the $T_\p$-operator for $\p\nmid \n$. 
\end{proof}

\begin{cor}
\label{Basis consisting of eigenforms of U_P operators}
The set of  $U_\p$-operators (for $\p \mid \n$) are simultaneously  diagonalizable on $S_{k,l}^{\new}(\Gamma_0(\n))$.
\end{cor}
\begin{proof}[Proof of Corollary~\ref{Basis consisting of eigenforms of U_P operators}]
Let $\p$ be a prime ideal of $A$ such that $\p \mid \n$.
By~\cite[Remark 2.17]{BV20}, the $U_\p$-operator is diagonalizable on $S_{k,l}^{\p-\new}(\Gamma_0(\n))$.
By Theorem~\ref{Oldspace and Newspace are stable under Hecke operators},
the space $S_{k,l}^{\new}(\Gamma_0(\n))$ is an $U_\p$-invariant subspace of $S_{k,l}^{\p-\new}(\Gamma_0(\n))$,
hence the $U_\p$-operator is also diagonalizable on $S_{k,l}^{\new}(\Gamma_0(\n))$. Now, the corollary follows
from Proposition~\ref{U_p, U_Q commutes and U_P, T_Q commutes} and the fact that a commuting set of diagonalizable operators on a finite dimensional vector space
are simultaneously diagonalizable.
\end{proof}

We conclude this article with a remark that $S_{k,l}^{\old}(\Gamma_0(\n))\cap S_{k,l}^{\new}(\Gamma_0(\n)) =\{ 0 \}$ may happen only for $l\ne 1$, because of the following proposition, which is in the spirit of Proposition~\ref{counter example for direct sum decomposition for square free p-part}. As a result, one may have to reformulate the definition of oldforms and newforms of level $\p\m$ for $l=1$.

\begin{prop}
\label{counter_example_type_1 for oldforms and newforms}
For any two distinct prime ideals $\p, \q$ generated by monic irreducible polynomials $P, Q$, respectively, the intersection
$S_{2,1}^{\old}(\Gamma_0(\p \q))\cap S_{2,1}^{\new}(\Gamma_0(\p\q))\ne \{0\}.$ Furthermore, for any $x\in \N$, 
$S_{2{q^x},1}^{\old}(\Gamma_0(\p \q))\cap S_{2{q^x},1}^{\new}(\Gamma_0(\p \q))\ne \{0\}$.
\end{prop}

\begin{proof}
We now show that   $E_{Q}- \delta_{P} E_{Q} \in S_{2,1}^{\old}(\Gamma_0(\p\q))\cap S_{2,1}^{\new}(\Gamma_0(\p \q)).$
By definition, $0\ne E_{Q}- \delta_{P} E_{Q} \in S_{2,1}^{\old}(\Gamma_0(\p \q))$.
From~\eqref{To show E_T-delta_p E_T is an old form trace 0} and~\eqref{To show E_T-delta_p E_T is an old form trace A-L 0}, we have
$ \Tr_{\q}^{\p \q}(E_{Q}- \delta_{P} E_{Q})=0=\Tr_{\q}^{\p \q}((E_{Q}- \delta_{P} E_{Q})|W_{\p}^{(\p \q)}).$ Since $W_{\p}^{(\p \q)}W_{\q}^{(\p \q)} = W_{\q}^{(\p \q)} W_{\p}^{(\p \q)}$, $W_{\p}^{(\p \q)} U_{\q}= U_{\q} W_{\p}^{(\p \q)}$, by Proposition~\ref{relation between Trace, A-L and U_P operator for level pm} and~\eqref{delta_P and A-L together}, we get
\begin{align*}
\Tr_{\p}^{\p \q}(E_{Q}- \delta_{P} E_{Q})
&=\Tr_{\p}^{\p \q}(E_{Q})- \Tr_{\p}^{\p \q}(E_{Q})|W_{\p}^{(\p \q)}  \\
&= \Tr_1^{\q}(E_{Q})- \Tr_1^{\q}(E_{Q})|W_{\p}^{(\p \q)}\quad (\mathrm{cf.~Corollary}~\ref{Tr of pm to m of level m})\\
&=0 \qquad (\mathrm{since}~ M_{2,1}(\GL_2(A))=0).
\end{align*}
Since $E_{Q}|W_{\q}^{(\p \q)}=E_{Q}|W_{\q}^{(\q)}=-E_{Q}$ (cf.~\cite[Proposition 3.3]{DK1})), we have
\begin{align*}
\Tr_{\p}^{\p \q}((E_{Q}- \delta_{P} E_{Q})|W_{\q}^{(\p \q)})
&= \Tr_{\p}^{\p \q}((E_{Q}|W_{\q}^{(\p \q)})-E_{Q}|W_{\p}^{(\p \q)}W_{\q}^{(\p \q)})\\
&=\Tr_{\p}^{\p \q}((E_{Q}|W_{\q}^{(\p \q)})-(E_{Q}|W_{\q}^{(\p \q)})|W_{\p}^{(\p \q)})\\
&=\Tr_{\p}^{\p \q}(-E_{Q}+(E_{Q}|W_{\p}^{(\p \q)})) 
=\Tr_{\p}^{\p \q}(-E_{Q}+\delta_{P} E_{Q})=0.
\end{align*}
Hence, $E_{Q}- \delta_{P} E_{Q} \in S_{2,1}^{\new}(\Gamma_0(\p \q))$. 
A similar argument shows that $0\ne E_{Q}^{q^x}-{{P}^{q^x-1}} \delta_{P} E_{Q}^{q^x} \in S_{2{q^x},1}^{\old}(\Gamma_0(\p \q))\cap S_{2{q^x},1}^{\new}(\Gamma_0(\p_1 \q))$ for $x\in \N$.
\end{proof}

\bibliographystyle{plain, abbrv}

\end{document}